\documentclass{llncs}
\usepackage{caption}
\usepackage{listings}
\usepackage{verbatim}
\usepackage{amsmath}
\usepackage{amssymb}
\usepackage{algorithm}
\usepackage{graphicx}
\usepackage[noend]{algpseudocode}
\usepackage{drum}
\usepackage{times}
\usepackage{multirow}
\usepackage[rightcaption]{sidecap}
\usepackage{url}
\usepackage{tcolorbox}
\usepackage{adjustbox}
\usepackage[font={small,it}]{caption}
\usepackage{fancyhdr}
\DeclareMathOperator*{\minimize}{min}

\usepackage[numbers]{natbib}
\usepackage{multicol}
\newtheorem{thm}{Theorem}

\DeclareCaptionFont{white}{ \color{white} }
\DeclareCaptionFormat{listing}{
  \colorbox[cmyk]{0.43, 0.35, 0.35,0.01 }{
    \parbox{\textwidth}{\hspace{15pt}#1#2#3}
  }
}
\author{Evan M. Drumwright}
\institute{Toyota Research Institute}

\graphicspath{{figs/}}

\begin{document}

\title{An Unconditionally Stable First-Order\\ Constraint Solver for Multibody Systems}

\maketitle
\begin{abstract}
This article describes an absolutely stable, first-order constraint solver for multi-rigid body systems that calculates (predicts) constraint forces for typical bilateral  and unilateral constraints, contact constraints with friction, and many other constraint types. Redundant constraints do not pose numerical problems or require regularization. Coulomb friction for contact is modeled using a true friction cone, rather than a linearized approximation. The computational expense of the solver is dependent upon the types of constraints present in the input. The hardest (in a computational complexity sense) inputs are reducible to solving convex optimization problems, i.e., polynomial time solvable. The simplest inputs require only solving a linear system. The solver is L-stable, which will imply that the forces due to constraints induce no computational stiffness into the multibody dynamics differential equations. This approach is targeted to multibodies simulated with coarse accuracy, subject to computational stiffness arising from constraints, and where the number of constraint equations is not large compared to the number of multibody position and velocity state variables. For such applications, the approach should prove far faster than using other implicit integration approaches. I assess the approach on some fundamental multibody dynamics problems.
\end{abstract}

\section{Introduction}
A constrained multi-body system is a collection of bodies subject to kinematic constraints, dynamic constraints, or both. Typical examples are robots, automobile suspensions, and containers of granules. The kinematics and dynamics of constrained multi-body systems  can be formulated explicitly, by use of judiciously selected coordinates, or implicitly, through the use of constraint forces that must be computed to satisfy certain algebraic equations. Depending on the model and the input (state and any control inputs to the model), accounting for the effect of those constraints on the dynamics is often very expensive and potentially results in computational stiffness in the application (simulation, trajectory optimization, etc.)  This article focuses on efficient solutions to initial value problems for multi-body systems subject to such implicit constraint equations, i.e., it focuses on simulating interesting multibody systems.

A constraint solver computes the aforementioned constraint forces (and through a simple linear operation, accelerations) acting on a multi-body system at a given state, under some modeling assumptions and using mathematical programming or nonlinear optimization algorithms (like solving complementarity problems).  The notion of a constraint solver arises from the observation that the dynamics of a constrained multibody system is coupled to its constraint forces but these problems can actually be decoupled into those of \1 computing the constraint forces via a constraint solver, i.e., solving an equation that looks like $G\inv{M}\tr{G}\lambda = -G\inv{M}f - \dot{G}v$ for $\lambda$ and \2 computing the multibody accelerations by solving an equation that looks like $M\dot{v} = \tr{G}\lambda + f$ for $\dot{v}$.

Constraint solvers are frequently used to help integrate multi-rigid body simulations with rigid constraints. They are used for modeling contact (and all other constraints) in engineering simulations, e.g., time-stepping based \texttt{Algoryx} and \texttt{Vortex} and event-based \texttt{SimWise 4D}.  Constraint solvers have been applied to similar problems in control, state estimation, and system identification applications.  Yet constraint solvers have suffered to date from problems both conceptual in the model (i.e., \emph{inconsistency} and \emph{indeterminacy}, inability to ``invert'' the model for optimizing control under constraints, non-smoothness of the model) and practical (i.e., \emph{constraint drift}, difficulty of solving complementarity problems quickly and robustly).

This article introduces a constraint solver that calculates constraint forces under the realistic model that constraints are all at least minimally compliant (specifically, that they are no more stiff than representable using double precision floating point). Admitting some compliance, even a nominal amount, into the existing ``hard'' constraint formulations causes the aforementioned conceptual and practical challenges to disappear and the problem instances to become computationally easy to solve.  And while \emph{computational stiffness} has obstructed the use of very stiff constraints for simulation, I will show how the present approach is not susceptible to this problem, by design. As an additional bonus, the constraint solver is L-stable: \emph{any} size integration step can be used to simulate a constrained system without the constraint forces giving rise to instability.

My constraint solver requires only one linear system solution or optimization per time interval; if optimization is required, the optimization problem is convex, even when modeling friction with a true cone, and the model does not require the discretization of the time interval to shrink for convergence.  And, that single optimization is relatively inexpensive and robust, since finding a feasible point is trivially accomplished by setting some variables to ``large'' and the rest to zero. Sparsity in the problem is common and readily exploited. The constraint solver enjoys several other computational advantages. It does the minimum ``collision detection'' work (i.e., a single evaluation) per timestep and offers up to first-order accuracy. It is easy to construct smooth constraints for input to the model and thereby ensure that the constraint solver is differentiable. Finally, the computational implementation is straightforward.

\subsection{Approach overview}
Softening the ``hard'' constraints for multibody systems into compliant constraints (like those in the penalty model) \emph{cause the constraints to become dynamical systems}. The constraints then couple motion to, and are influenced by, a multi-body system. One way to model the static and dynamic behavior of these compliant constraints is with multiple spring-dampers. Accordingly, I will describe a stable and efficient numerical approach that:
\begin{enumerate}
\item Given deformations $\phi(q)$ and deformation rates $\dot{\phi}(q, v)$ for the spring-dampers (i.e., the constraint system state variables) as functions of multibody state $\{ q(t), v(t) \}$ at time $t$,
\item computes the spring-damper forces implicitly using \emph{future} deformation and deformation rates (i.e., those at $t + h$, for small $ h$),
\item and applies these computed spring-damper forces in a fast, explicit manner to integrate the multibody state to $\{ q(t+h), v(t+h) \}$.
\end{enumerate}
I will show how the implicit Euler approach is particularly efficient at integrating multi-spring-damper systems that are computationally stiff and potentially computationally stiff.  I will also show how decoupling the multi-body and spring-mass systems, which makes their independent solution possible, still yields solutions with only second-order truncation error.

I will not describe how to model physical phenomena using the constraints. I assume that the user can define and evaluate spring-damper type constraint functions. But I will also demonstrate how various kinds of constraints can be modeled: joint constraints, range-of-motion limit constraints, and contact constraints---using both Hertzian contact and elastic foundation models---with Coulomb friction.

\subsection{Focus application: multi-body simulation}
This article will focus on the application of the constraint solver to the simulation problem:  compute $q(t + h)$, $v(t + h)$ and constraint forces $\lambda(t)$ given $q(t)$, $v(t)$, and $h$. These data are a superset of the necessary data for control, state estimation, and system identification purposes. Control applications, for instance, might require only computing $\lambda(t)$ if $q(t+h)$ and $v(t+h)$ may be obtained to high accuracy from sensor measurements. In any case, efficient application of the constraint solver to simulation implies efficient application to the control, state estimation, and system identification domains as well. We will see throughout this article that simulations using the constraint solver are most efficient---compared to other simulation approaches, like applying implicit integrators to the normally stiff differential equations arising from constraints---when the number of constraints is not large compared to the number of multibody position and velocity state variables. \S\ref{section:unilateral-msd} will show a pathological counterexample (900 constraints vs. 13 state variables), in which the best-case performance of the constraint solver was not as good as optimized rivals, but \emph{its worst-case performance was far better} than theirs.

\subsection{Outline}
The layout of this article follows. Section~\ref{section:related-work} surveys related work in modeling and solving constraints and the use of these models in control, state estimation, system identification, and simulation applications. Section~\ref{section:spring-mass-damper} examines the spring-mass-damper system, which is the foundational model for the present work, and considers various first-order approaches for integrating the system dynamics. Section~\ref{section:holonomic-constraints} uses the stable first-order approach examined in Section~\ref{section:spring-mass-damper} to integrate the dynamics of a multi-rigid body system constrained using only bilateral constraints. Section~\ref{section:unilateral-msd} extends this approach to a unilaterally constrained system (e.g., a robot hitting a range-of-motion limit). Section~\ref{section:modeling-contact} extends the approach  yet again, this time to contact constraints with friction. Section~\ref{section:discussion} closes with concluding thoughts.

\section{Related work}
\label{section:related-work}
This section surveys two existing kinds of constraint solvers: ones that model constraint forces as a function of deformation (``penalty methods'', \S\ref{section:related-work:penalty}) and ones that compute forces to satisfy kinematic and force constraints (\S\ref{section:related-work:rigid-constraints}). \S\ref{section:related-work:constraint-applications} additionally surveys applications of constraint solvers to control, state estimation, system identification, and simulation domains.

\subsection{Penalty methods}
\label{section:related-work:penalty}
There is a vast literature covering compliant contact models, which are sometimes known as penalty methods. In the context of compliant contact models, deformation results in a proportional force. In the context of optimization via penalty methods, deformation can be viewed as the violation of a constraint, and solving the resulting DAE by treating it as an ODE with virtual penalization forces is effected using a kind of dynamic optimization. These virtual forces cause the constraint violations to be minimized dynamically (over time) in the absence of loads; in other words,   
the behavior of the system as it is stabilizing can be viewed analogously to the penalty method class of optimization algorithm. Summarizing, there are two very different approaches that both effect compliant contact. And the constraint solver introduced in this article draws from both approaches by modeling contact from first principles and by leveraging constrained optimization.  

The contact aspects of the constraint model introduced in this article are effectively that of the Song et al.~\cite{Song:2000a} contact model, though I introduce numerous computational and modeling advantages over the approach they describe. Mine does not require categorizing contact into sticking and sliding. My approach permits using realistic stiffness values derived via engineering tables and without needing to consider whether the choice of modeling parameters might make the simulation unstable.

\subsection{Rigid constraints and rigid contact}
\label{section:related-work:rigid-constraints}
Approaches that compute forces to satisfy kinematic constraints and force constraints have traditionally been applied to modeling \emph{rigid constraints}. Perfectly rigid constraints often simplify multi-rigid body models in several respects, all of which are conferred through coordinate reduction: revolute and prismatic joints, to name but two examples, use interpretable, computationally efficient, and kinematically simple independent parameters. But the concept of rigid constraints can introduce some theoretical and practical challenges as well. As just one example, constraint forces can be distributed arbitrarily if care is not taken when constraints are both rigid and redundant; Zapolsky and Drumwright had to work around this particular problem in the design of a constraint-aware inverse dynamics controller for quadrupedal robots~\cite{Zapolsky:2015c}. The constraint solver introduced in this article is also able to model bilateral constraints, and it allows them to be ``rigid'' up to full floating point precision; put another way, constraints can be modeled as rigid to a greater extent than is observed in nature. For example, the Young's Modulus of diamond is on the order of a terapascal, which only requires twelve of the approximately sixteen decimal digits of precision provided by 64-bit floating point numbers, assuming that the pascal is the chosen unit.

A ``rigid [point] contact'', i.e., a point of contact between two completely rigid bodies, corresponds to a unilateral rigid constraint and several more equations that describe the frictional interaction at that point (see, e.g.,~\cite{Trinkle:1997}). Rigid contact makes reasoning about the contact geometry simple, but the problems from redundant forces described above (also called ``indeterminacy'') remain. The conceptual problem of ``inconsistent configurations'', like Painlev\'{e}'s famous paradox (described in detail in~\cite{Stewart:2000a}) emerges with rigid contact also. And rigid contact is non-smooth: force is discontinuous when two bodies first contact, for example. In addition to these challenges that arise from the mathematics (and likely the real physics as well, see, e.g.~\cite{Collins:2001}), the computational demands of implementing such models are significant.

Specifically, it is challenging along multiple dimensions to solve the complementarity problems that emerge from the rigid contact formulations. Solutions are known to be non-unique.\footnote{The effect of that non-uniqueness on multi-body dynamics software is unknown, and that lack of knowledge is troubling given the potentially large number of software implementations using this approach for engineering purposes.} Existence proofs only exist for problem approximations (i.e., discretization of the continuous time dynamics, polygonal friction cones). Even for the approximate problems that are known to possess solutions (\cite{Stewart:1996,Anitescu:1997}) that can be found with an algorithm (i.e., solving linear complementarity problems with Lemke's Algorithm), the worst-case time complexity is exponential. Additionally, the necessary algorithms do not lend themselves to robust numerical implementations: it is hard to guarantee solutions to tight accuracy (satisfying the non-negativity and complementarity conditions to tight tolerances), particularly when the complementarity problems become ``degenerate'' from redundancy in the constraints (see, e.g.,~\cite{Cottle:1992}, p.~340). Without accuracy guarantees, it is challenging for the solver to even identify when it has solved the problem.

In addition to these theoretical and practical challenges, implementation introduces several more wrinkles: \1 Coulomb friction requires selecting a typically arbitrary tolerance for distinguishing between sliding and rolling/not-sliding. \2 The rigid contact model requires distinguishing between impacting and sustained contact (using another arbitrary tolerance); the former mode requires applying an appropriate multi-body \emph{impacting constraint} model (all constraints, including contact constraints with friction, range of motion constraints, etc., are potentially active during an impact, so an impact model that accounts for all of these is necessary). \3 Considerable thought is necessary to identify the minimum ``states'' (as in the sense of a finite state machine; these kinds of states are typically called ``modes'') for a multi-point rigid constraint model, and the logic to switch between the various modes is incredibly complex. I'm unaware of a peer-reviewed description of the necessary logic in existing literature.

Drumwright and Shell~\cite{Drumwright:2010b} and Todorov~\cite{Todorov:2011} both described modified rigid contact models toward alleviating these challenges. Unfortunately, neither model was capable of producing solutions consistent with non-impacting (i.e., ``sustained'') contact mechanics, which requires complementarity. As already noted, the model described in this article is based on an established compliant model of contact that has not been controversial. Still, my approach allows the model to be effectively rigid, i.e., rigid to the 16 decimal digits that IEEE 754 double precision allows, but without any of the aforementioned computational headaches. Additionally, this newly introduced approach is easier to solve than the models from~\cite{Drumwright:2010b,Todorov:2011}, both of which require solving a positive semi-definite (PSD) linear complementarity problem (or, equivalently, a PSD quadratic program) to compute an initial feasible point.

\subsection{Applications of constraint solvers}
\label{section:related-work:constraint-applications}
Constraint solvers have been applied in multiple domains, namely control, state estimation, system identification, and simulation. This section surveys applications of constraint solvers in these areas. Note that a constraint solver need not exist as an explicit code module; the constraint solver for a simulation with compliant contact can be considered to be the combination of the contact model and the ODE initial value problem solver (i.e., the integrator).

\subsubsection{Constraint solvers for control}
\label{section:related-work:constraint-models-for-control}
Constraint solvers for control have been investigated heavily in the context of robotics under the name ``whole body control''. The basic formula has been to control a robot subject to contact constraints; other kinds of constraints can usually be incorporated into these formulations, but contact has been the traditional focus. Representative works include those of Todorov~\cite{Todorov:2011}, Escande et al.~\cite{Escande:2012}, Feng et al.~\cite{Feng:2015}, and Zapolsky and Drumwright~\cite{Zapolsky:2015}. Rigid contact and Coulomb friction models are typically used, implying either a differential algebraic equation or a differential complementarity problem mathematical model that is a function of control inputs $u$, e.g.:
\begin{align}
\dot{x} & = f(x, u) \\
0 & = g(x)
\end{align}
Inverting this model---computing $u$ given $x$ and $\dot{x}$---is computationally expensive when $f(.)$ and $g(.)$ incorporate contact constraints. The inverse is not required to be unique~\cite{Zapolsky:2015}.

The presented constraint solver does not require regularization (cf.~\cite{Todorov:2011}), does not assume rigidity (cf.~\cite{Escande:2012,Kuindersma:2015,Feng:2015}), can use a true friction cone (cf.~\cite{Zapolsky:2015}), is always solvable (i.e., inconsistent constraints are not of concern) with a unique solution, and does not violate physical laws (cf.~\cite{Todorov:2011}).

\subsubsection{Constraint solvers for state estimation}
\label{section:related-work:constraint-models-for-estimation}
Estimating state (position and velocity) of contacting rigid bodies has been investigated in works by Zhang et~al. \cite{Zhang:2013b}, Pollard~et~al.~\cite{Pollard:2013}, Kuindersma~et~al.~\cite{Kuindersma:2015}, and Kumar~et~al.~\cite{Kumar:2016}. The promise is to use contact mechanics (or a gross approximation thereof) to compensate for the poor observability of object state while objects are being manipulated. Estimating state for contacting bodies has proven to be a particularly difficult problem, and the deficiencies in prior process models is just one of the challenges to be overcome~\cite{Koval:2015}; Koval~et~al. note ``particle starvation'' as one such significant challenge~\cite{Koval:2015}. However, it is clear that better (faster, more accurate) process models would be beneficial to state estimation; Zhang~et~al.'s complementarity-problem-based state estimator was shown to be too slow for real-time, even when applied to a robotic hand modeled with only three point contacts.

\subsubsection{Constraint solvers for system identification}
\label{section:related-work:constraint-models-for-sysid}
System identification is the process of estimating parameters for a plant (i.e., the ``process model'' or ``transition model'' in the context of state estimation) from telemetry data. Constraint solvers can be employed in the system identification process to estimate parameters of the contact model in particular (see, e.g.,~\cite{Fazeli:2017b} where impact restitution is estimated and~\cite{Kolev:2015,Fazeli:2017b} where the friction coefficient is estimated).  The presented constraint solver provides smooth gradients that are useful for the optimization process used for system identification.

\subsubsection{Constraint solvers for simulation}
\label{section:related-work:constraint-models-for-simulation}
Constraint solvers for simulation exhibit particular challenges not present in the aforementioned domains. Compliant models tend to lead to \emph{computational stiffness}, which limits simulation step sizes for reasons of stability rather than accuracy. Simulations with rigid contact constraints either use an event-driven scheme or are only zeroth-order accurate (i.e., the ``velocity stepping'' approaches), and simulations using rigid contact models tend to generate many events. Commercial software products have spanned all three of the approaches I described above: MSC's \texttt{Adams} is known to use a compliant model, \texttt{SimWise 4D} is known to use an event-driven scheme with a rigid contact model, and \texttt{Havok} is known to use the zeroth-order rigid contact approach.

\subsection{Summary}
We have now seen how constraint solvers can be used in control, state estimation, system identification, and simulation applications, and we have examined the dichotomy of constraint solvers: ``soft'' (compliant) and ``hard'' (rigid). We will next examine the simple mass-spring-system that serves as the physical model of our soft constraints.

\section{The Spring-Mass-Damper system}
\label{section:spring-mass-damper}
 
The foundational model of all constraints in this article is that of the mass-spring-damper system:
\begin{align}
m\ddot{x} + b\dot{x} + kx = f.
\end{align}
Various types of forces can be modeled by using either the dynamics of this system or through the equivalent relationship:
\begin{align}
\lambda = -kx -b\dot{x}.
\end{align}
As just one example, the mass-spring-damper readily emulates linearly elastic contact, an oft reasonable model.

The following subsections will investigate how to integrate multibody systems that are subject to such constraints (i.e., have forces applied from such a forcing model) under a parameter range of $k$ that would correspond to springs constructed from various materials. As examples, longitudinal forces acting on a steel beam with cross-sectional area $1\ \textrm{m}^2$ and $1$ m length would exhibit a stiffness ($k$ in N/m) equivalent to the Young's Modulus of steel ($\approx 10^{12}$~Pa). The same beam made from stiff rubber might have a Young's Modulus of $10^{5}$~Pa. So, $k$ must be able to take on a rather large range. Setting damping/dissipation values in a principled way from material properties is an open problem: damping/dissipation is currently used to achieve a qualitative effect, even in accuracy-critical FEM codes. But it seems reasonable to require that the viscous damping parameter lies in the range $b \in [0 \frac{\textrm{N} \cdot \textrm{s}}{\textrm{m}}, 10^{12} \frac{\textrm{N} \cdot \textrm{s}}{\textrm{m}}]$.

We next examine how various first-order schemes integrate this spring-mass-damper system forward in time. If a particular integration scheme is incapable of simulating the spring-mass-damper system under the targeted parameters with reasonably large step sizes, that same integrator will be unable to simulate a multi-body system subject to dynamic constraints with similar stiffnesses.

\subsection{Stable first-order integration of a spring-mass-damper system}
\label{section:stable-msd-integration}
The canonical first-order formulation of the mass-spring-damper system is:
\begin{align}
v & = \dot{x} \\
\dot{v} & = \frac{f - bv - kx}{m}.
\end{align}
The various Euler integration schemes simply represent various ways to discretize $x(t_1)$ and $v(t_1)$ to a first-order approximation of the continuous time dynamics around $x(t_0)$ and $v(t_0)$, where $t_1 = t_0 + h$. For shorthand, we'll write $x_0$, $v_0$, etc. to refer to a particular state variable at a particular point in time.

\subsubsection{Explicit Euler}
The explicit Euler update rule takes the following form:
\begin{align}
x_1 & = x_0 + h v_0 \\
v_1 & = v_0 + \frac{h}{m} (f -bv_0 - kx_0),
\end{align}
meaning that for $f = 0$ the problem can be treated as a linear difference equation:
\begin{align}
y_1 & = Ay_0,
\end{align}
where $y \equiv \tr{\begin{bmatrix} x & v \end{bmatrix}}$ and the \emph{iteration matrix} is defined as:
\begin{align}
A \equiv \begin{bmatrix}
1 & h \\
\frac{-hk}{m} &\ 1 - \frac{hb}{m}
\end{bmatrix}.
\end{align}
The iteration matrix's eigenvalues shows that the approach is not stable for \emph{any} step size (for stability, the magnitudes of all eigenvalues, real and imaginary, must be less than unity). Figure~\ref{fig:explicit-euler} shows that explicit Euler is clearly unstable for $h = 10^{-5}, k = 10^6, b = 1$, somewhat more stable for $h = 10^{-5}, k = 10^6, b = 10$, stable for $h = 10^{-5}, k = 10^6, b = 2 \times 10^5$, and unstable for $h = 10^{-5}, k = 10^6, b = 2.1 \times 10^5$. For the range of parameters we wish to use for $k$ and $b$, explicit Euler is clearly unsuitable. 

\begin{figure}[htpb]
    \centering
\includegraphics[width=.625\linewidth]{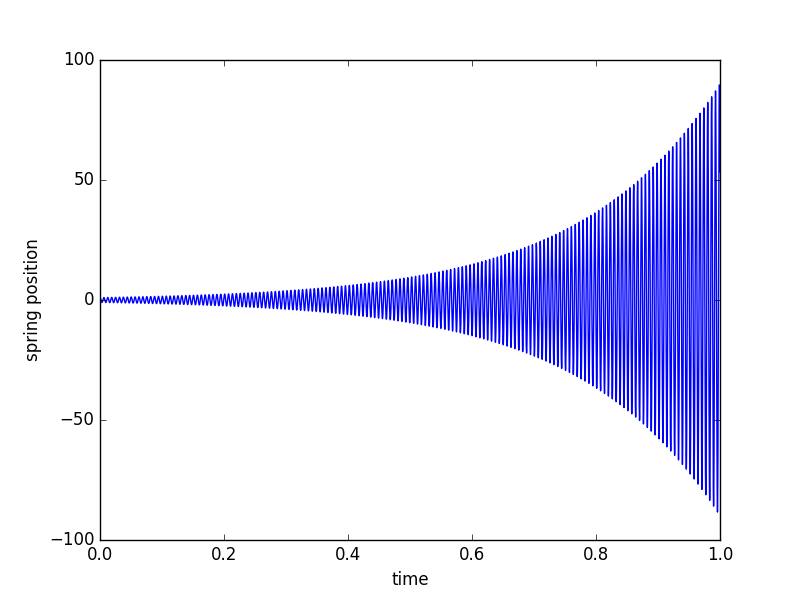}
    \includegraphics[width=.625\linewidth]{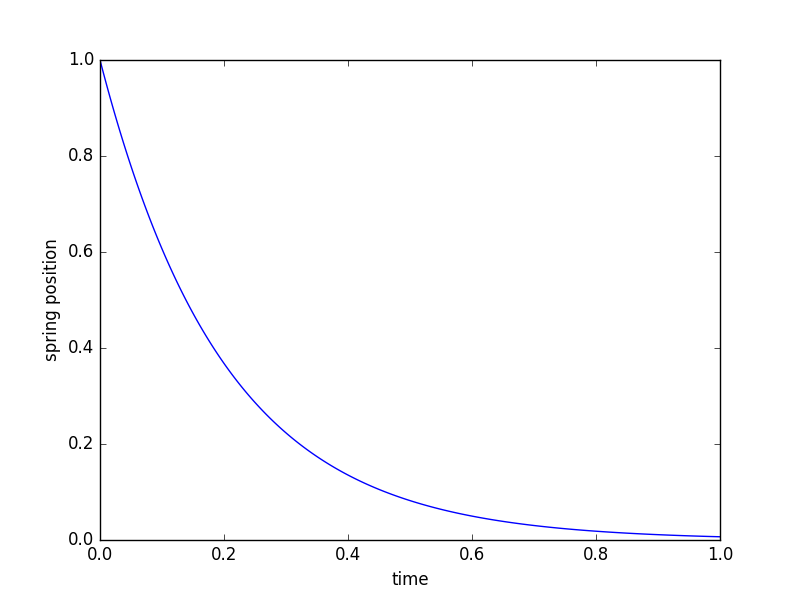}
    \includegraphics[width=.625\linewidth]{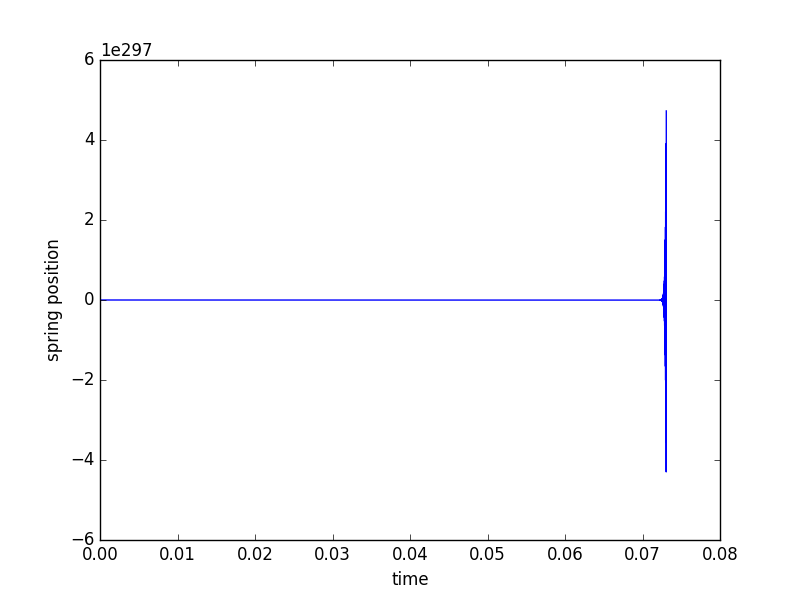}
    \caption{Explicit Euler integration of a spring-mass-damper system with $f = 0, h=10^{-5}$, $m = 1$, $k = 10^6$ for $b = 1$ (top), $b = 2 \times 10^5$ (middle), and $b = 2.1 \times 10^5$ (bottom). Explicit Euler is not stable under the targeted range of conditions.}
    \label{fig:explicit-euler}
\end{figure}

\subsubsection{Semi-implicit Euler}
From~\cite{Hairer:2006} (1.9), the following update rule yields the semi-implicit Euler integrator, which is symplectic:
\begin{align}
x_1 & = x_0 + h v_1 \\
v_1 & = v_0 + \frac{h}{m} (f -bv_1 - kx_0).
\end{align}
Symplectic integrators have been applied to problems extensively for their momentum conserving properties. However, while semi-implicit Euler is stable for any non-negative value of $b$, it is not stable for any non-negative $k$, as Figure~\ref{fig:semi-implicit-euler} shows.

\begin{figure}[htpb]
    \centering
    \includegraphics[width=.495\linewidth]{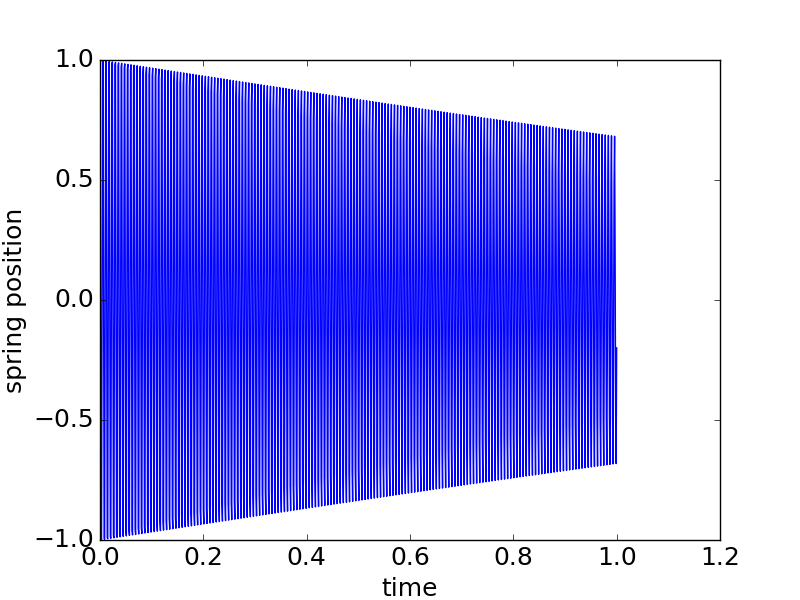}
    \includegraphics[width=.495\linewidth]{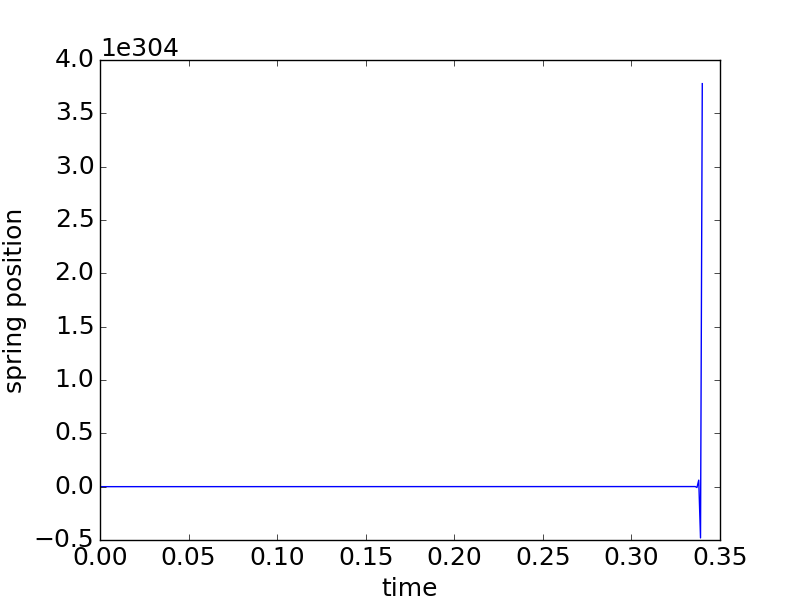}
    \caption{Semi-implicit Euler integration of a spring-mass-damper system with $f = 0, h=10^{-3}$ , $m = 1$, and $b=1$ for $k=10^6$ (left) and $k=10^7$ (right). Semi-implicit Euler is only stable when $k$ is sufficiently small.}
    \label{fig:semi-implicit-euler}
\end{figure}

\subsubsection{Implicit Euler}
The implicit Euler update rule
\begin{align}
x_1 & = x_0 + h v_1 \\
v_1 & = v_0 + \frac{h}{m} (f -bv_1 - kx_1)
\end{align}
is well known to be stable for any step size, and Figure~\ref{fig:implicit} illustrates this behavior. This integration scheme typically introduces significant computational requirements \emph{for arbitrary systems}, however, as determining the next states typically requires solving nonlinear systems of equations. The next section will show how the approach introduced in this article is able to avoid this computationally hard problem.

\begin{figure}
    \centering
    \includegraphics[width=.625\linewidth]{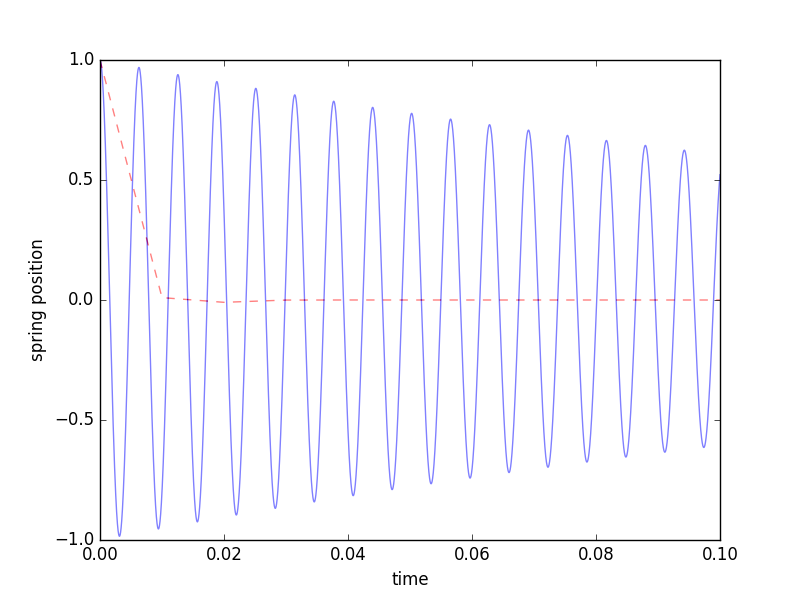}
    \caption{Implicit integration of the mass-spring system with $f = 0, k=10^{6}$, $m = 1$, and $b = 0$ for $h=10^{-5}$ (blue) and $h=10^{-2}$ (red, dashed). The integration is stable for any step size, though it is clear that the solutions are not particularly accurate---the amplitude should not attenuate.}
    \label{fig:implicit}
\end{figure}

\subsection{Conclusion}
For integrating mass-spring-damper systems with $k$ and $b$ in our target range, a stiff integration technique is necessary. But implicit integrators may be avoided in multibody dynamics simulations because the requisite nonlinear equation solutions can prove prohibitively expensive (especially when bodies are contacting). The artificial dissipation introduced by implicit methods (illustrated in Figure~\ref{fig:implicit}) is often a lesser  concern.

\section{Application of the solution of the mass-spring-damper system to multibody systems with holonomic constraints}
\label{section:holonomic-constraints}
The last section showed that a stiff integration (implicit) technique permits stable integration of mass-spring-damper systems with the necessary parameter ranges. This section will 
introduce a new DAE formulation for multibody systems subject to spring-damper forcing models (or, equivalently, subject to kinematic constraints) and will show how to integrate this system (i.e., solve initial value problems for the DAE formulation) by building on the implicit Euler algorithm described in the previous section. 


Consider a multi-body system having $m$ generalized coordinates $q(t)$ which are constrained by $n$ holonomic functions $\phi_i(t, q)$, for $i \in \{ 1,\ldots,n \}$ (the subscript will only be called out when necessary hereafter). The usual way that this constraint is modeled is as ``hard'' in the sense that solutions are sought that satisfy $\phi(t, q) = 0$ to high accuracy (generally somewhat less than full double precision). Without loss of generality, I will ignore the version of the constraint that is an explicit function of time, i.e., $\phi(q) = 0$. Now imagine that the constraint is no longer to be modeled as hard, where its ``violation'' should be zero, but instead acts as a virtual spring and damper that produces a force $\lambda$ that penalizes deviation from rest:
\begin{align}
\lambda = - K\phi(q) - B\dot{\phi}(q, v) - \hat{M}\dot{G}v, \label{eqn:lambda-first-def}
\end{align}
for some choice of non-negative diagonal matrices $B \in \mathbb{R}^{n \times n}$ and $K  \in \mathbb{R}^{n \times n}$, and where the variables $G$ (and, by extension, its time derivative) and $\hat{M}$ will be defined shortly. This viewpoint of a constraint could be useful for simulating a four bar linkage by virtually welding two rigid bodies together. Or, we might wish to model a contact patch between two bodies using springs with stiffnesses $K$ and damping $B$. In fact, this equation is very close to how penalty methods are applied in multi-body simulation, in which initial value problems for the following ordinary differential equation are solved: 
\begin{align}
\dot{q} & = Nv \\
M\dot{v} & = \tr{G}\lambda + f(t, q, v) \\
\lambda & = -K\phi(q) - B\dot{\phi}(q, v) \label{eqn:penalty-method-lambda},
\end{align}
where $v(t)$ are now the system generalized velocities, $M(q)$ is the generalized inertia matrix, $f(t, q, v)$ is every force but constraint forces acting on the multi-body system, $N(q)$ is a left-invertible matrix that maps elements in $v$ to elements in $\dot{q}$, and $G(q) = \frac{\partial \phi}{\partial q}N$. This latter variable comes from taking the time-derivative of $\phi(q): \dot{\phi} = \frac{\partial \phi}{\partial q} \dot{q}$ and using exactly the relationship we just described, $\dot{q} = Nv$. $G$ also yields linear mappings from generalized velocities to constraint velocities and (through transposition) constraint forces to generalized forces.\footnote{This duality arises from mechanical power, which relates velocity and force.} The astute reader will note that the $\hat{M}\dot{G}v$ term from \eqref{eqn:lambda-first-def} is absent in \eqref{eqn:penalty-method-lambda}. I'm showing how forces are applied in the context of the penalty method in \eqref{eqn:lambda-first-def}, and I will shortly provide reasoning for the inclusion of the mysterious $\hat{M}\dot{G}v$ term.

\subsection{Explicitly coupling the multibody and spring-damper systems}
The system of ODEs above does not explicitly consider the dynamics of the constraint (i.e., its trajectory over time), but we will by requiring that it acts as a mass-spring-damper system:
\begin{align}
\hat{M}\ddot{\phi}(q, v, \dot{v})  + B\dot{\phi}(q, v) + K\phi(q) = \hat{f}. \label{eqn:constraint-dynamics}
\end{align}
where $\hat{M}$ is a particular kind of inertia (I call it the ``constraint-space inertia matrix'') and $\hat{f}$ is a force in constraint space. This system has a single fixed point at $\{ \phi(q) = 0, \dot{\phi}(q, v) = 0\}$. I will show that the following DAE, which combines the system of ODEs above with (\ref{eqn:constraint-dynamics}), permits computing efficient solutions to initial value problems for $q(t)$ and $v(t)$:
\begin{align}
\dot{q} & = Nv \label{eqn:dotq} \\
M\dot{v} & = \tr{G}\lambda + f(t, q, v) \label{eqn:dotv} \\
0 & = g(t, q, v, \dot{v}) = \hat{M}\ddot{\phi}(q, v, \dot{v}) + B(q, v)\dot{\phi}(q, v) + K(q)\phi(q) - \hat{f} \label{eqn:g} \\
\lambda & = -K(q)\phi(q) - B(q, v)\dot{\phi}(q, v) - \hat{M}\dot{G}v \label{eqn:lambda-first},
\end{align}
where:
\begin{align}
\hat{M} & \equiv \inv{(G\inv{M}\tr{G})} \label{eqn:hatM} \\
\hat{f} & \equiv \hat{M}G\inv{M}f. \label{eqn:hatf}
\end{align}

Note that $K$ and $B$ have been ``promoted'' from constant terms (in the penalty method) to state-dependent terms, allowing, e.g., stiffnesses to be configuration dependent (which is necessary for correctly modeling linear elasticity) and contact forces to be smooth (by making $B$ a function of $\phi(q)$ as in~\cite{Hunt:1975}, thereby eliminating a possible discontinuity around $\phi = 0$). I will demonstrate how $K$ and $B$ can be computed using proper units in different applications (range-of-motion limits, elastic foundation contact, non-conforming contact) in examples examined throughout this article. 

Why this DAE? It explicitly declares the coupled dynamical systems: the multibody system and the spring-mass-damper system. The spring-mass-damper system is expressed in a form where the parameters ($K$ and $B$) are easily interpretable and its behavior readily analyzed. As just two simple examples, we could identify the undamped oscillation frequency or damping ratio.

Now let us examine the DAE applied to a particle, with mass $m$ and state $(x \in \mathbb{R}, \dot{x}  \in \mathbb{R})$, that is compliantly constrained to the environment using a virtual mass-spring-damper: $\hat{M}\ddot{\phi} + B\dot{\phi} + K\phi = \hat{f}$. Consider how the system of ODEs in (\ref{eqn:dotq})--(\ref{eqn:lambda-first}) is simplified when $m \in \mathbb{R}$ replaces $M$, $1$ (unity) replaces $N$, $x$ replaces $q$, $\dot{x}$ replaces $v$, and $\phi(x) \equiv x$:
\begin{align}
m\ddot{x} =& f + \lambda \label{eqn:first} \\
\hat{f} =& \hat{M}\ddot{\phi} + B\dot{\phi} + K\phi \\
\lambda =& -K\phi -B\dot{\phi}.
\end{align}
Since $\partial \phi/\partial q = G = 1$, we have replaced $\tr{G}\lambda$ in (\ref{eqn:dotv}) with just $\lambda$ and $\hat{M}\dot{G}v$ with zero in \eqref{eqn:lambda-first}. From $\phi(x) \equiv x$ we can obtain $\dot{\phi} = \dot{x}$ and $\ddot{\phi} = \ddot{x}$. It should now be apparent that the particle is equivalent to a spring-mass-damper with stiffness $K$ and damper $B$ that is affixed to the environment when $\hat{M} = m$ and $\hat{f} = f$. 

Similarly, we can attach a rigid body system to the environment using springs and dampers and seek the same effect (with respect to stiffness, damping, and inertia). To bridge the ODE for $\phi$ (Equation~\ref{eqn:g}) with the ODEs for $q$ and $v$ (Equations~\ref{eqn:dotq} and~\ref{eqn:dotv}), we use the observation $\dot{\phi} = \frac{\partial \phi}{\partial q} \dot{q} = \frac{\partial \phi}{\partial q} N v$ to define $G \equiv  \frac{\partial \phi}{\partial q} N$, implying $\dot{\phi} = Gv$ (it can be shown that $G$ is the transpose of the force transmission matrix) and:
\begin{align}
\ddot{\phi} = G\dot{v} + \dot{G}v, \label{eqn:ddot-phi}
\end{align}
as well. Rearranging \eqref{eqn:dotv}:
\begin{align}
\dot{v} = \inv{M}(\tr{G}\lambda + f),
\end{align}
we can substitute $\dot{v}$ into \eqref{eqn:ddot-phi}, yielding:
\begin{align}
\ddot{\phi} = G\inv{M}\tr{G}\lambda + G\inv{M}f + \dot{G}v.
\end{align}
We now substitute the definition of $\lambda$ from \eqref{eqn:lambda-first}:
\begin{align}
\ddot{\phi} = G\inv{M}\tr{G}(-B\dot{\phi} -K\phi - \dot{G}v) + G\inv{M}f + \dot{G}v.
\end{align}
Moving all $\phi$ terms and derivatives to the left hand side:
\begin{align}
\ddot{\phi} + G\inv{M}\tr{G}(B\dot{\phi} + K\phi - \hat{M}\dot{G}v) = G\inv{M}f + \dot{G}v,
\end{align}
and noting that we want the equation to look like \eqref{eqn:g}, we multiply both sides by $\hat{M}$ to arrive at:
\begin{align}
\hat{M}\ddot{\phi} + \hat{M}G\inv{M}\tr{G}(B\dot{\phi} + K\phi + \hat{M}\dot{G}v) = \hat{M}G\inv{M}f + \hat{M}\dot{G}v.
\end{align}
Using the prior definitions $\hat{M} = \inv{(G\inv{M}\tr{G})}$ and $\hat{f} = \hat{M}G\inv{M}f$, we can simplify this to:
\begin{align}
\hat{M}\ddot{\phi} + B\dot{\phi} + K\phi = \hat{f}.
\end{align}
In other words, the rigid body system acts like a mass-spring-system affixed to the environment given $\lambda$ as defined in \eqref{eqn:lambda-first}. Also note that \eqref{eqn:g} is redundant in the DAE: it follows from the definitions of $\lambda$, $\hat{M}$, and $\hat{f}$ as shown above.

I will now show that it is possible to compute the evolution of the constraint dynamics   \emph{independently of $q_1$ and $v_1$} with a first-order solution. This will allow me to show that a solution method can leverage the computationally easy problem of \emph{implicitly} advancing the state of a spring-damper system with the also easy problem of \emph{explicitly} (or semi-explicitly) advancing the state of a multi-rigid body system, all while remaining (at least) first-order accurate. \emph{The upshot is that the method is efficient and stable whatever the magnitudes of $K$ and $B$.}

\subsection{First-order scheme: formulating multibody constraint problems with purely bilateral constraints}
\label{section:formulating-bilateral-constraints}
The challenge with applying an ODE initial value problem solver (i.e., integrator) to \eqref{eqn:dotq}--\eqref{eqn:lambda-first} directly is that the equations are prone to introducing computational stiffness (when $K$ or $B$) is large. The method of choice for computationally stiff ODEs is an implicit integrator, as we observed in Section~\ref{section:spring-mass-damper}, but implicit integrators are usually inefficient when the equations aren't stiff; explicit integrators are the method of choice in that case. The solution to this efficiency problem would seem to be a scheme that can switch between the two approaches (implicit and explicit) as appropriate. Such automatic stiffness detection has rarely proven to be efficient in practice (see~\cite{Hairer:1996}, pp. 21--24).

Instead, I will now describe a particular first-order scheme for the previous constrained multibody system, including the necessary sequence of operations. This particular scheme was constructed to have particular qualities. First, the ODEs for the constraint variables are solved implicitly (for stability). Second, the ODEs for $q$ and $v$, i.e., \eqref{eqn:dotq} and \eqref{eqn:dotv}, are solved semi-explicitly ($v$ is computed explicitly, then that value is used to compute $q$), which avoids solving nonlinear systems of equations and minimizes creep for sticking friction (when friction is introduced into the model, in Section~\ref{section:modeling-contact}). The scheme follows.

Given $q_0$, $v_0$, $N_0$, $M_0$, $G_0$, $f(q_0, v_0)$, $\phi_0$, and $\dot{\phi}_0$, compute:
\begin{align}
q_1 & = q_0 + h N_0 v_1 + O(h^2) \label{eqn:q1} \\
v_1 & = v_0 + h M_0^{-1}(\tr{G_0}\lambda_1 + f_0) + O(h^2)  \label{eqn:v1} \\
\phi_1 & = \phi_0 + h \dot{\phi}_1 + O(h^2) \label{eqn:phi1} \\
\dot{\phi}_1 & = \dot{\phi}_0 + h \ddot{\phi}_0 + O(h^2). \label{eqn:varphi1}
\end{align}
As this section has moved the discussion from the realm of algebra to numerical implementation, further consideration of the $_0$ and $_1$ notation is necessary. Subscript-$0$ denotes a quantity to be evaluated as an input to the problem. On the other hand, subscript-$1$ denotes a variable to be solved for; these variables are what the constraint solver computes. We want to have as few subscript-$1$ variables as possible, and those variables should not be ``buried'' within a function; consider that e.g., \eqref{eqn:q1} specified instead as $q_1 = q_0 + h N_1 v_1 + O(h^2)$ would require solving a nonlinear system of equations for $q_1$ (which we will avoid since many collision detection calls would be necessary in evaluating $N_1$). Finally, variables \emph{not} specified with a subscript will not be evaluated; these variables will be used solely to simplify complicated expressions.

My accuracy analysis will make use of the following two lemmas. 
\begin{lemma}
A function of a time-dependent vector (e.g., $f(x(t))$) that can be well approximated by a Taylor Series around $x(t)$ can be approximated with linear error without consideration of any Jacobian terms. 
\end{lemma}

\begin{proof}
The proof is evident using only the first two terms of the Taylor Series expansion:
\begin{align}
f(x(t+h)) = f(x(t)) + h \frac{\partial f}{\partial x} \dot{x} + O(h^2). \label{eqn:taylor}
\end{align}
And therefore
\begin{align}
f(x(t+h)) = f(x(t)) + O(h).
\end{align}
\end{proof}

\begin{lemma}
Any time-dependent term in an ODE evaluated at time $t_0$ that is approximated by its value at $t_0+h$ introduces $O(h^2)$ truncation error into the solution of the initial value problem at $t_0+h$.   
\end{lemma}

\begin{proof}
Consider the first two terms in the Taylor Series, $f(x(t_0+h)) = f(x(t_0)) + h \dot{f}(x(t_0), \dot{x}(t_0)) + O(h^2)$; this is equivalent to ~\eqref{eqn:taylor} because $\dot{f} = \frac{\partial f}{\partial x} \dot{x}$.  Evaluating $x(t_0)$, to pick one term arbitrarily, at $t_0+h$ instead of $t_0$ introduces an error of $O(h)$ into $\dot{f}$'s evaluation, from Lemma~1: 
\begin{align}
f(x(t_0+h)) = & f(x(t))|_{x(t_0)} + h (\dot{f}(x(t), \dot{x}(t))|_{x(t_0+h), \dot{x}(t_0)} + O(h)) + \ldots \\
& \quad O(h^2). \nonumber
\end{align}
The $h O(h) = O(h^2)$ term is subsumed by the existing $O(h^2)$ term. It should be clear that any combination of substitutions of $x(t)$ and $\dot{x}(t)$ by their values at $x(t+h)$ and $\dot{x}(t+h)$, respectively, results in identical truncation error.
\end{proof}

\begin{thm}
Integrating $\phi(q)$ and $\dot{\phi}(q, v)$ independently from $q$ and $v$, all from initial time $t_0$, yields solutions at $t_1$ for $\phi(q_1)$, $\dot{\phi}(q_1, v_1)$, that exhibit quadratic truncation error in $h = t_1 - t_0$.
\end{thm}

\begin{proof}
The proof will start from the explicit Euler algorithm applied to $\phi$ and $\dot{\phi}$. Explicit Euler is known to be locally first-order accurate (equivalent to saying that the truncation error is quadratic). Quadratic truncation error will be maintained through every step. Explicit Euler applied to $\phi$ and $\dot{\phi}$ yields the equations:
\begin{align}
\phi_1 & = \phi_0 + h\dot{\phi}_0 + O(h^2) \label{eqn:explicit-euler-phi} \\
\dot{\phi}_1 & = \dot{\phi}_0 + h\ddot{\phi}_0 + O(h^2). \label{eqn:explicit-euler-dotphi}
\end{align}
This approach can be transformed into a semi-explicit Euler method by using the identity $\dot{\phi}_0 = \dot{\phi}_1 + O(h)$ and applying Lemma~2 to~\eqref{eqn:explicit-euler-phi}:
\begin{align}
\phi_1 & = \phi_0 + h\dot{\phi}_1 + O(h^2) \label{eqn:semi-explicit-euler-phi}
\end{align}
Thus,~\eqref{eqn:semi-explicit-euler-phi} and~\eqref{eqn:explicit-euler-dotphi} represent the foundation of the integration scheme. We now demarcate evaluation times for the ``bridging equation'', \eqref{eqn:ddot-phi}:
\begin{align}
\ddot{\phi}_0 = G_0\dot{v}_0 + \dot{G}_0v_0 \label{eqn:ddotphi0},
\end{align}
and for~\eqref{eqn:dotv} as well, rearranging slightly first:
\begin{align}
\dot{v}_0 = \inv{M_0}(\tr{G_0}\lambda_0 + f_0). \label{eqn:Mvdot0}
\end{align}
Applying Lemma~1 to this equation allows us to establish the relationship:
\begin{align}
\dot{v}_0 = \inv{M}_0(\tr{G_0}\lambda_1 + f_0) + O(h). \label{eqn:dotv0_star}
\end{align}
Then substituting~\eqref{eqn:dotv0_star} into~\eqref{eqn:ddotphi0} yields:
\begin{align}
\ddot{\phi}_0 & = G_0\underbrace{\big(\inv{M_0}(\tr{G_0}\lambda_1 + f_0) + O(h)\big)}_{\dot{v}_0} + \dot{G}_0v_0.    
\end{align}
And substituting \emph{that} into~\eqref{eqn:explicit-euler-dotphi} gives:
\begin{align}
\dot{\phi}_1 & = \dot{\phi}_0 + h(\underbrace{G_0(\inv{M_0}(\tr{G_0}\lambda_1 + f_0)) + \dot{G}_0v_0)}_{\ddot{\phi_0}} + O(h^2), \label{eqn:varphi1-prime}
\end{align}
from Lemma~2.

Combining~(\ref{eqn:phi1}) and~(\ref{eqn:varphi1-prime}) and an additional equation for $\lambda_1$ (Equation~\ref{eqn:lambda-first} with specified times for $K$, $B$, $\phi$ and $\dot{\phi}$) yields three blocks of equations in three blocks of unknowns ($\phi_1$, $\dot{\phi}_1$, and $\lambda_1$):
\begin{align}
\phi_1 & = \phi_0 + h \dot{\phi}_1 + O(h^2) \label{eqn:3x3-first} \\
\dot{\phi}_1 & = \dot{\phi}_0 + h\big(G_0(\inv{M_0}(\tr{G_0}\lambda_1 + f_0)) + \dot{G}_0v_0\big) + O(h^2) \label{eqn:3x3-middle} \\
\lambda_1 & = -K_0\phi_1 - B_0\dot{\phi}_1 - \inv{(G_0\inv{M}_0\tr{G}_0)}\dot{G}_0v_0 + O(h). \label{eqn:lambda-bilateral}
\end{align}
Note the provenance of the $O(h)$ term for the last equation comes from mixing evaluation times, in accordance with Lemma~1. We can solve the 3 block $\times$ 3 block linear system above by first substituting \eqref{eqn:lambda-bilateral} into \eqref{eqn:3x3-middle}, yielding:
\begin{align}
\dot{\phi}_1 & = \dot{\phi}_0 +  \ldots \\ 
& \quad h\big(G_0(\inv{M_0}(\tr{G_0}\underbrace{(-K_0\phi_1 - B_0\dot{\phi}_1 - \inv{(G_0\inv{M}_0\tr{G}_0)}\dot{G}_0v_0 + O(h))}_{\lambda_1} + \ldots \nonumber \\
& \qquad f_0)) + \dot{G}_0v_0\big) + O(h^2) \nonumber 
\end{align}
$G_0\inv{M}_0\tr{G}_0\inv{(G_0\inv{M}_0\tr{G}_0)}$ yields the identity matrix, leaving us with $-h\dot{G}_0v_0$ and $h\dot{G}_0v_0$ terms, which cancel. Also, the $hO(h)$ term is subsumed by the $O(h^2)$ term, allowing us to remove it. Combining both simplifications, we arrive at:
\begin{align}
\dot{\phi}_1 & = \dot{\phi}_0 + h\big(G_0(\inv{M_0}(\tr{G_0}(-K_0\phi_1 - B_0\dot{\phi}_1 + f_0))\big) + O(h^2).
\end{align}
We next substitute \eqref{eqn:3x3-first} into this equation, yielding:
\begin{align}
\dot{\phi}_1 & = \dot{\phi}_0 + h\big(G_0(\inv{M_0}(\tr{G_0}(-K_0\underbrace{(\phi_0 + h\dot{\phi}_1 + O(h^2))}_{\phi_1} - B_0\dot{\phi}_1) + f_0)) + O(h^2).
\end{align}
Now the $hO(h^2)$ term is subsumed by the $O(h^2)$ term:
\begin{align}
\dot{\phi}_1 & = \dot{\phi}_0 + h\big(G_0(\inv{M_0}(\tr{G_0}(-K_0(\phi_0 + h\dot{\phi}_1) - B_0\dot{\phi}_1) + f_0)) + \big) + O(h^2).
\end{align}
When we move all $\dot{\phi}_1$ terms to the left hand side, we arrive at:
\begin{align}
& (I + h\big(G_0(\inv{M_0}(\tr{G_0}(hK_0 + B_0)))\big)\dot{\phi}_1 = \dot{\phi}_0 + \ldots\\ 
& \qquad h\big(G_0(\inv{M_0}(\tr{G_0}(-K_0\phi_0 + f_0))\big) + O(h^2).  \nonumber 
\end{align}
Collecting terms and simplifying using the following definitions:
\begin{align}
\hat{M}_0 & \equiv \inv{(G_0\inv{M_0}\tr{G_0})} \\
Y & \equiv \inv{(I + \inv{\hat{M}_0}(h^2K_0 + hB_0))},
\end{align}
we obtain:
\begin{tcolorbox}
\begin{align}
\dot{\phi}_1 & = Y(\dot{\phi}_0 - h\inv{\hat{M}_0}K_0\phi_0 + hG_0\inv{M_0}f_0) + O(h^2) \label{eqn:dotphi1-bilateral} \\ 
\phi_1 & = \phi_0 + h\dot{\phi}_1 + O(h^2). \label{eqn:phi1-bilateral}
\end{align}
\end{tcolorbox}
The computation of $\lambda_1$ (via Equation~\ref{eqn:lambda-bilateral}) permits evaluating $v_1$ (via Equation~\ref{eqn:v1}), which in turn permits evaluating $q_1$ (via Equation~\ref{eqn:q1}). No step of the scheme has given up beyond quadratic truncation error for the state variables ($q$, $v$, $\phi$, and $\dot{\phi}$). Therefore, \eqref{eqn:dotphi1-bilateral} and~\eqref{eqn:phi1-bilateral} yield first-order accurate solutions $\phi(q_1)$ and $\dot{\phi}(q_1, v_1)$ without accurate knowledge of the value of $q_1$ and $v_1$. $\blacksquare$
\end{proof}

\noindent\makebox[\linewidth]{\rule{\linewidth}{0.4pt}} \vspace{2.5mm}

\textbf{For the remainder of this article, I will continue to assume that $f$ and $B$ are evaluated using $q_0, v_0$ and $N$, $G$, $M$ and $K$ are evaluated using $q_0$, allowing us to drop the subscripts used previously. Thus $f \equiv f_0$, $N \equiv N_0$, $G \equiv G_0$, $\dot{G} \equiv \dot{G}_0$, $K \equiv K_0$, $B \equiv B_0$, $M \equiv M_0$, $\hat{f} \equiv \hat{f}_0$ and $\hat{M} \equiv \hat{M}_0$.}
\\

\vspace{2.5mm}
\noindent\makebox[\linewidth]{\rule{\linewidth}{0.4pt}}

\subsection{Implications}
Conceptually, the first-order scheme described in the previous subsection provides an algorithm for integrating two coupled sets of dynamical systems---the multibody system and the constraints---independently and with first-order accuracy.

This result is profound: it means that stiff, or even rigid, holonomic constraint forces can be computed efficiently using only the solution to a linear system. No nonlinear system of equations needs to be solved, which also means that step sizes do not need to be limited to ensure convergence of a Newton-Raphson scheme (cf.~\cite{Hairer:1996}, pp. 123--126). The resulting constraint forces can then be used with at least any first-order integration scheme; extension to higher order schemes might also be possible. 

Also note that the matrix $Y$ is always invertible. It is the sum of a positive definite matrix ($I$, the identity matrix) and the product of a positive semi-definite matrix and a linear combination of positive semi-definite, diagonal matrices ($K$ and $B$). Given that $\tr{u}Iu > 0$ and $\tr{u}(h(hK + B))u \geq 0$, $Y$ must be positive definite as well, because
\begin{align}
\tr{u}(I + h(hK + B))u > 0
\end{align}
for any vector $u$ of nonzero real values. Positive definite matrices are always invertible (in fact, their inverses are positive definite as well). The implication of the invertibility of $Y$ is that constraint forces become distributed among the constraints: indeterminacy is resolved organically. The approaches for unilateral and contact constraints will share this nice property also. 

Finally we address the question of how such stability is attained for so little computational cost. The key is that the forces due to the constraints do not cause computational stiffness; it is still possible to introduce computational forces through other forces (e.g., large PD control terms) applied to the multibody system. Instead, the constraint-based approach removes any computational stiffness from the constraints.

\subsection{Example: pendulum}
The constraint-based approach was used to simulate a pendulum in 2D using an absolute coordinate formulation. The configuration of the bob was represented using three-coordinates: center-of-mass location (two coordinates, $x$ and $y$) and bob orientation (one coordinate, $\theta$). The velocity variables were simply $\dot{x}, \dot{y}$, and $\dot{\theta}$. The holonomic constraint was:
\begin{align}
\phi = \begin{bmatrix}
x \\
y
\end{bmatrix} + 
R(\theta)u,
\end{align}
where $R(\theta)$ is an orientation matrix that transforms vectors in the bob's body frame to the world frame and $u$ is a vector from the bob to the nominal joint location. The pendulum was started in the fully horizontal position with zero velocity and was simulated for one second of virtual time. At every time step, (\ref{eqn:dotphi1-bilateral}) and~(\ref{eqn:phi1-bilateral}) were used to determine the next $\dot{\phi}$ and $\phi$, respectively, which were used in turn to compute, via~(\ref{eqn:lambda-bilateral}), $\lambda_1$. $\lambda_1$ was used in turn to compute the next $q$ and $v$ (via the symplectic scheme represented by Equations~\ref{eqn:v1} and \ref{eqn:q1}).

Figure~\ref{fig:constraint-deviation-pendulum} shows the constraint deviation under large stiffnesses for both a large ($h=0.1$) and small ($h=10^{-4}$) step sizes; note that these ``errors'' scale quadratically with $h$, consistent with the $O(h^2)$ error proved in Section~\ref{section:formulating-bilateral-constraints}. The symplectic scheme caused the energy of the pendulum to be maintained to precisely zero. Observers who have experience with implicit integration techniques, which numerically effect some artificial dissipation, might be wondering how the scheme conserves energy for large step sizes. The explanation is that momentum \emph{orthogonal to the constraint} is conserved. An example in the next section will show that system energy indeed dissipates whenever momentum is not orthogonal to the constraints.  

\begin{figure}[htpb]
\includegraphics[width=.495\linewidth]{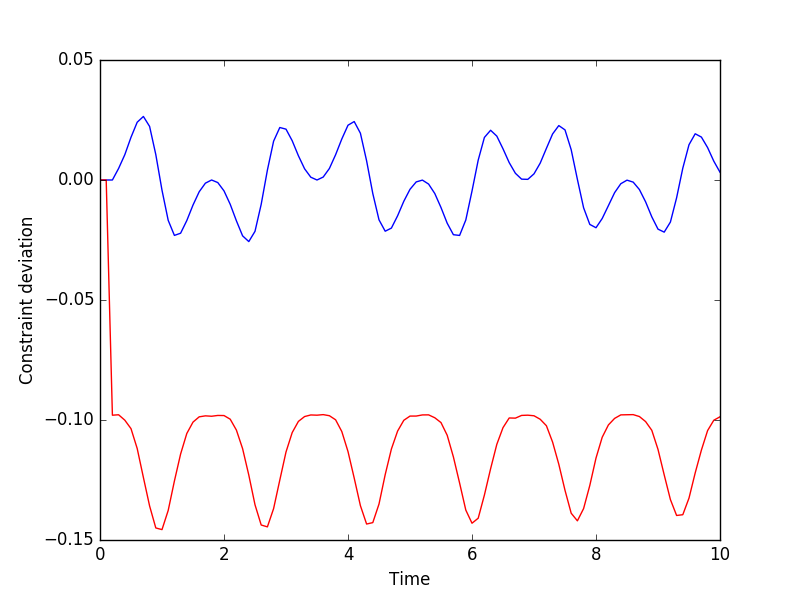}
\includegraphics[width=.495\linewidth]{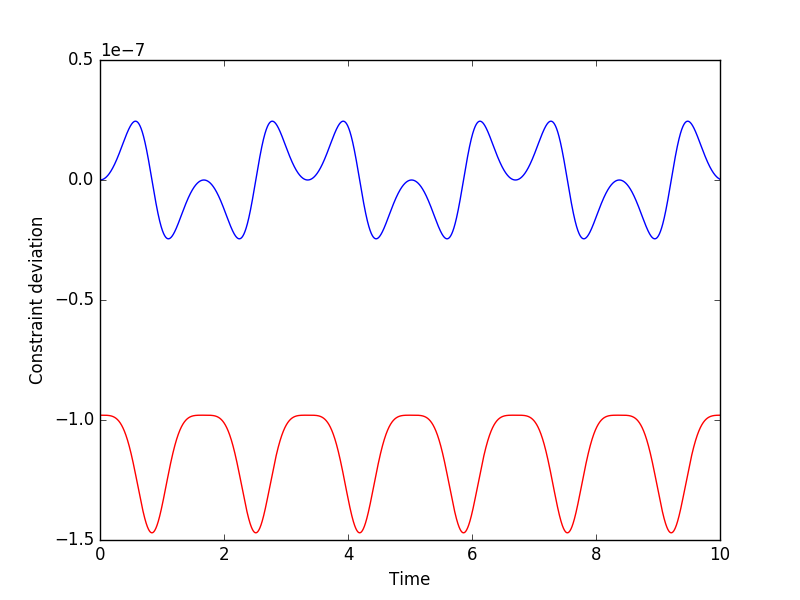}
\caption{Deviation of the pendulum from its nominal joint position under the parameters $k = 10^{15}$ and $b = 1$ for $h = 0.1$ (left) and $h = 10^{-4}$ (right) . While the stiffness is large, the system does not become unstable. Note the different scales in the two plots. \label{fig:constraint-deviation-pendulum}}
\end{figure}

\section{Unilaterally constrained spring-damper model}
\label{section:unilateral-msd}
While bilateral constraints are useful to introduce the methodology underlying my approach, the hard technical challenges usually emerge in multibody dynamics only after unilateral constraints---and the complementarity problems used to model the dynamics---are introduced. And so I will now demonstrate how we can apply my approach to multibody systems subject to unilateral constraints. Consider the following system composed of a  spring-damper affixed to the ``world'' that can compress arbitrarily but not extend past its resting length, and a particle that can move only along one dimension. When the particle enters the region $x \leq 0$, the spring is considered to be compressed. The spring can only push (i.e., not pull) against the particle (i.e., only compressive forces are allowed). When the spring is at its resting length (implied by $x > 0$), it can impart no force on the particle. Figure~\ref{fig:unilateral-msd} depicts these conditions.

\begin{figure}
    \centering
    \includegraphics[width=.33\linewidth]{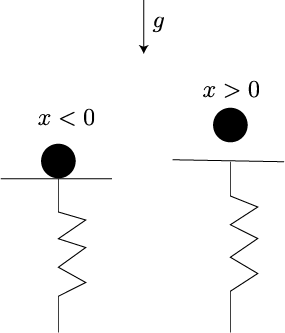}
    \caption{An illustration of the unilateral spring system modeling contact between a particle with height $x$ and ``the world''; we are interested in how the spring-damper system imparts forces on the particle. We assume that the spring can only compress; it cannot extend past its resting length. The surface of the spring-damper system can only impart forces on the particle when the spring is compressed (left); it is assumed both that the spring is at its resting length and that the rate of change of the resting length is zero when $x > 0$ (right). The figure also shows the direction of gravity.}
    \label{fig:unilateral-msd}
\end{figure}

\subsection{Holonomic unilateral constraint formulation for multibody systems}
Consider the multibody system with state $(q, v)$, vector function of $n$ holonomic, unilateral constraint functions $\phi(q) \to \mathbb{R}^n$, generalized inertia matrix $M$, Jacobian matrix G = $\frac{\partial \phi}{\partial q} N$, and external force vector $f(q, v)$.  Stiffness and damping scalars $k$ and $b$ again become non-negative diagonal matrices $K \in \mathbb{R}^{n \times n}$ and $B  \in \mathbb{R}^{n \times n}$.

I use a piecewise DAE model\footnote{A piecewise DAE is a sequence in time of (typically related) differential algebraic equations, where the initial conditions for the $i^{\textrm{th}}$ DAE are drawn from the solution to the $(i-1)^{\textrm{th}}$ DAE, and the initial conditions for the first DAE are specified by the user.}, where the ``pieces'' correspond to times where a constraint becomes active (i.e., where $\phi$ is non-positive \emph{and} $\lambda$ transitions from zero to positive) or inactive (i.e., where $\lambda$ transitions from positive to zero). No better than first-order accuracy can be obtained without an event-detection mechanism that isolates times in which constraints transition from active to inactive and vice versa~\cite{Anitescu:2004a}.

I will reformulate this DAE as a \emph{differential programming problem} (DPP), which combines ordinary differential equations with mathematical programming variables, where solution to the latter is necessary to solve initial value problems over the former. The mathematical programming problem in my particular formulation is a convex quadratic program for which the solution (to $\lambda$) yields the maximum value of zero and $-K\phi(q) - B\dot{\phi}(q, v) - \hat{M}\dot{G}v$, obviating an otherwise necessary discontinuous nonlinear equality constraint (namely $\lambda = \max{(0, -K\phi(q) - B\dot{\phi}(q, v) - \hat{M}\dot{G}v)}$ for unknown $\lambda, \phi$, and $\dot{\phi}$). 

\subsection{DAE model of a unilaterally constrained system}
The version of the DAE from~(\ref{eqn:dotq})--(\ref{eqn:lambda-first}) modified to model a purely compressive, compliant unilateral constraint is given below: 
\begin{align}
\dot{q} = &\ Nv \\
M\dot{v} = &\ \tr{G}\lambda + f \label{eqn:throwaway2} \\
\hat{f} = &\ \hat{M}\ddot{\phi} + B\dot{\phi} + K\phi \label{eqn:unilateral-constraint-dynamics} \\
\lambda = & \max{(0, -K\phi - B\dot{\phi} - \hat{M}\dot{G}v}). \label{eqn:nonlineq}
\end{align}
The nonlinear equality constraint in~(\ref{eqn:nonlineq}) makes this DAE very challenging to solve. Nonlinear systems of equations are challenging to solve in general, as feasible points can be hard to find and success is usually dependent upon having a good starting point. And~(\ref{eqn:nonlineq}) is not differentiable at $\lambda = 0$, stymieing typical derivative-based search strategies. However, the DAE above is equivalent to the following DPP that computes $\lambda$ via a convex quadratic program:
\begin{align}
\minimize_\lambda & \frac{1}{2}\tr{\lambda}\lambda \\
\dot{q} = &\ Nv \nonumber \\
M\dot{v} = &\ \tr{G}\lambda + f \nonumber \\
\hat{f} = &\ \hat{M}\ddot{\phi} + B\dot{\phi} + K\phi \nonumber \\
\lambda \geq &\ -K\phi - B\dot{\phi} - \hat{M}\dot{G}v \\
\lambda \geq &\ 0.
\end{align}
where recall that we have defined:
\begin{align*}
\hat{M} \equiv &\ \inv{(G\inv{M}\tr{G})} \\
\hat{f} \equiv &\ \hat{M}G\inv{M}f
\end{align*}
in \eqref{eqn:hatM} and \eqref{eqn:hatf}, respectively. Computing the minimum value of a number that is at least as large as two other numbers is simply a maximization operation. In contrast to the previous DAE that required solving a nonlinear system of equations, the quadratic program above is trivial to solve.

This particular DPP will also prove useful when I introduce contact constraints with friction. For now, we analyze the convexity and feasibility of this problem when it is used to solve the time-discretized version of this problem for $\{ \lambda, \phi_1, \dot{\phi}_1 \}$. 

\subsection{The time-discretized DPP for a unilaterally constrained system}
\label{section:formulating-generic-unilateral-constraints}
Discretizing in time, removing \eqref{eqn:unilateral-constraint-dynamics}---these dynamics follow from the definition of $\lambda$, so \eqref{eqn:unilateral-constraint-dynamics} is redundant---and ignoring $v_1$ and $q_1$ momentarily yields:
\begin{align}
\minimize_{\lambda_1, \phi_1, \dot{\phi}_1}\ & \frac{1}{2}\tr{\lambda_1}{\lambda_1} \\
\phi_1 & = \phi_0 + h \dot{\phi}_1 + O(h^2) \\
\dot{\phi}_1 & = \dot{\phi}_0 + h \underbrace{(G\dot{v}_0 + \dot{G}v_0)}_{\ddot{\phi_0}} + O(h^2) \label{eqn:throwaway1} \\
\lambda_1 & \geq - K\phi_1 - B\dot{\phi}_1 - \hat{M}_0\dot{G}v + O(h) \label{eqn:second-to-last-inequality} \\
\lambda_1 & \geq 0,
\end{align}
where the $O(h)$ term in \eqref{eqn:second-to-last-inequality} comes from Lemma~1, and we use \eqref{eqn:ddotphi0} in \eqref{eqn:throwaway1}. Simplifying \eqref{eqn:throwaway1} further, we obtain:
\begin{align}
\minimize_{\lambda_1, \phi_1, \dot{\phi}_1}\ & \frac{1}{2}\tr{\lambda_1}{\lambda_1} \nonumber \\
\phi_1 & = \phi_0 + h \dot{\phi}_1 + O(h^2) \nonumber \\
\dot{\phi}_1 & = \dot{\phi}_0 + h (G\underbrace{\inv{M}(\tr{G}\lambda_1 + f  + O(h))}_{\dot{v}_0} + \dot{G}v_0) + O(h^2) \\
\lambda_1 & \geq - K\phi_1 - B\dot{\phi}_1 - \hat{M}\dot{G}v + O(h) \nonumber \\
\lambda_1 & \geq 0, \nonumber
\end{align}
from \eqref{eqn:throwaway2}, where the $O(h)$ term again comes from Lemma~1. As before, that term is scaled by $h$ and thus becomes subsumed by the existing $O(h^2)$ term:
\begin{align}
\dot{\phi}_1 & = \dot{\phi}_0 + h (G\inv{M}(\tr{G}\lambda_1 + f) + \dot{G}v_0) + O(h^2).
\end{align}


The QP can be simplified by substituting out $\phi_1$ and $\dot{\phi}_1$ in~(\ref{eqn:second-to-last-inequality}). Starting with the former,
\begin{align}
\lambda_1 & \geq - K\underbrace{(\phi_0 + h \dot{\phi}_1 + O(h^2))}_{\phi_1} - B\dot{\phi}_1 - \hat{M}\dot{G}v_0 + O(h),
\end{align}
and then collecting terms yields:
\begin{align}
\lambda_1 & \geq - K\phi_0 - (hK + B)\dot{\phi}_1 - \hat{M}\dot{G}v_0 + O(h).
\end{align}
Then, replacing $\dot{\phi}_1$ results in
\begin{align}
\lambda_1 & \geq - K\phi_0 - (hK + B)\underbrace{(\dot{\phi}_0 + h \inv{\hat{M}}\lambda + hG\inv{M}f + h\dot{G}v_0 + O(h^2))}_{\dot{\phi}_1} - \ldots \nonumber \\
& \qquad \hat{M}\dot{G}v + O(h).
\end{align}
The result after collecting terms and removing the subsumed $hO(h^2)$ term is:
\begin{align}
& (I + (h^2K + hB)\inv{\hat{M}})\lambda_1 \geq \\
& \qquad  - K\phi_0 - (hK + B)(\dot{\phi}_0 + hG\inv{M}f + h\dot{G}v_0) - \hat{M}\dot{G}v_0 + O(h). \nonumber
\end{align}
The purely inequality constrained QP below (with asymptotic error terms removed) is the ultimate result:
\begin{tcolorbox}
\begin{align}
\minimize_{\lambda_1}\ & \frac{1}{2}\tr{\lambda_1}{\lambda_1} \\
(I + (h^2K + hB)\inv{\hat{M}})\lambda_1 & \geq - K\phi_0 - (hK + B)(\dot{\phi}_0 + \ldots \nonumber \\ 
& \qquad \quad  hG\inv{M}f  + h\dot{G}v_0) - \hat{M}\dot{G}v_0 \\
\lambda_1 & \geq 0.
\end{align}
\end{tcolorbox}
$v_1$ and $q_1$ are then computable using~(\ref{eqn:v1}) and~(\ref{eqn:q1}), respectively. This QP is always feasible: simply increasing $\lambda_1$ until all inequality constraints are satisfied yields a feasible point. And the Hessian of the objective function for this QP is the identity matrix, which means that the QP is strictly convex. $\phi_1$ and $\dot{\phi}_1$ are not computed directly in this formulation; contrast with the approach for bilateral constraints described in the previous section.

\subsection{Example: pendulum with soft range-of-limit}
Figure~\ref{fig:unilateral-pendulum} shows the result of this strategy applied to a pendulum system, now using a minimal coordinate $( \theta, \dot{\theta} )$ representation; assuming that the positive $y$-axis points upward, $\theta = 0$ is defined such that the pendulum bob configuration is closest to $y = -\infty$. The pendulum is subject to a soft range of motion limit, imposing $\theta \geq 0$. This example system was simulated from the initial conditions $\{ \theta = \frac{\pi}{2}, \dot{\theta} = 0 \}$.

Figure~\ref{fig:unilateral-pendulum} shows that this system can be simulated stably with large step sizes; as expected, accuracy increases linearly as $h$ decreases.

\begin{figure}[htb]
    \centering
    \includegraphics[width=.95\linewidth]{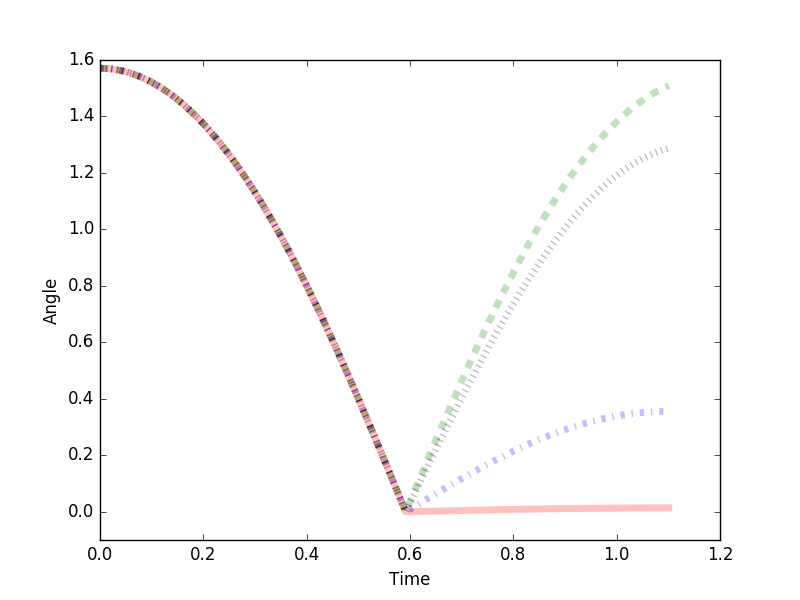}
    \caption{A plot of the pendulum with range of limit constraint at $\theta = 0$, with $k=10^{12}$ and $b=0$ simulated with $h=10^{-5}$ (red, solid), $h=10^{-6}$ (blue, dashed-dotted), $h=10^{-7}$ (grey, dotted) and $h=10^{-8}$ (green, dashed). While the system can be simulated stably for \emph{any} value of $h$, the true, elastic behavior is not evident until the step size becomes small; for values $h \geq 10^{-5}$, the impact against the range-of-motion limit appears to be ``dead''.}
    \label{fig:unilateral-pendulum}
\end{figure}

\subsection{Example: box on an elastic foundation}
\label{section:elastic-foundation-example}
As a second example \emph{that will illustrate pathological worst-case performance for our method}, consider a rigid cube (with side length $s=1$m and density $8$kg$/$m$^3$) resting on an elastic foundation (e.g., Winkler model) of depth 1m with a variable number of elements (\{ 100, 900 \}). Assume that the contact is frictionless. The system is simulated for one second of virtual time, starting from the box resting (zero velocity) on top of the elastic foundation at full extension (no deformation). I modeled the foundation as very stiff, using Young's Modulus values of $E = 10^{11}$ N/m$^2$ (about that of steel) and $E = 10^7$ N/m$^2$ (about that of rubber) and a damping value of $B = I \frac{\textrm{N} \cdot \textrm{s}}{\textrm{m}}$. Note that increasing the number of elastic elements from 100 to 900 decreases the stiffness ($K$) of each element in the foundation. $K$ is easily determined by multiplying the Young's Modulus by the (uniform) element area and dividing by the depth of the foundation.   

I additionally applied a force (of variable magnitude \{ 5N, 50N \}) to the top of the cube at the point $s/2 \tr{\begin{bmatrix} \cos{(10t)} & \sin{(10t)} & 1 \end{bmatrix}}$ (in the cube frame). The force is directed downward (i.e., it is always in the direction of gravity). The constraint-based approach was tested against MATLAB's \texttt{ode15s} implicit integrator, which I set to run in first-order for a equal comparison. I also validated the results using MATLAB's \texttt{ode45}, which provides a high accuracy solution. I used a step size of $0.01$s for both my constraint-based approach and MATLAB's integrators. 

Table~\ref{table:constraint-vs-ode15s} summarizes the timing results, which I will now discuss . First, the $5$N force results in a steady-state solution, which \texttt{ode15s} is able to exploit. Indeed, the integrator requires essentially the same computational time to simulate to the designated end time (1s) or nearly any time beyond that. On the other hand, the $50$N force does not yield a steady state solution: the applied force causes a rocking behavior where the side of the cube opposite the point of application to separate slightly from the foundation. The result is that the \texttt{ode15s} runs much slower on this scenario.

Table~\ref{table:constraint-vs-ode15s} also highlights the disparity in running times when the number of elements is increased. The constraint-based approach scales cubically with the number of elements. The running time of \texttt{ode15s} is dominated by a $O(m^2) f(n)$ factor to form the necessary Jacobian matrix, where $m$ is the number of velocity variables and $f(n)$ is the running time of the collision detection algorithm. The collision detection algorithm for this particular example is relatively inexpensive and runs in linear time in the number of elements. The reason that implicit integrators like \texttt{ode15s} are fast at integrating steady state problems is because the Jacobian matrices can be reused. For the non-steady state problem, the Jacobians may require regular reconstruction (which requires significant computation), in which case the constraint-based method may be faster even when many elastic elements are used.

I note that the constraint-based method works reasonably well on all problems. It does not run for thousands of seconds on any of the examples. Its running time suffers from neither computational stiffness nor rapidly changing dynamics; it scales predictably as the number of constraints increase. Additionally, we can imagine re-using Jacobian matrices, applying warm starting, and employing other computational tricks that MATLAB's mature ODE integrators exploit for solving initial value problems.  

\begin{table}[htpb]
    \centering
    \begin{tabular}{|c|c|c|c|c|c|c|c|}
    \hline
\multicolumn{2}{|c|}{} & \multicolumn{3}{c|}{$E = 10^7$  N/$m^2$} & \multicolumn{3}{c|}{$E = 10^{11}$ N/$m^2$} \\ \hline
elastic & load & constraint-based & \texttt{ode15s} & \texttt{ode45} & constraint-based & \texttt{ode15s} & \texttt{ode45} \\ 
elements & & approach & && approach && \\ \hline
100 & 5N & 0.331s & 0.302s & 28.0s & 0.286s & 0.189s & 27.4s \\ \hline
100 & 50N & 0.332s & 9.55s & 15.5s & 0.396s & 5550s & 621s \\ \hline
900 & 5N & 7.40s & 4.75s & 829s & 8.23s & 1.10s & 17100s \\ \hline
900 & 50N & 12.98s & 27.3s & 346s & 26.00s & 2.31s & 17400s \\ \hline
    \end{tabular}
    \caption{Timings (in seconds) to simulate the cube-on-elastic-foundation scenario with both the constraint-based approach and other integration techniques. }
    \label{table:constraint-vs-ode15s}
\end{table}

The error of the constraint-based approach with $h = 0.01$ on the elastic foundation problem is provided in Table~\ref{table:error}. I compared final states against those obtained from $\texttt{ode45}$: it's locally accurate to fourth-order and uses error control, and I have treated it as the ground truth. Absolute and relative errors are computed after one second of simulation time. These errors represent the Euclidean norm of the generalized coordinates (position and orientation) of the rigid body. The first four rows of the table show that absolute and relative error are quite good when the scenario is not particularly dynamic (follows from little loading). The last four rows of the table show significant error, though not instability.

\begin{figure}[t]
    \centering
    \includegraphics[width=.7\linewidth]{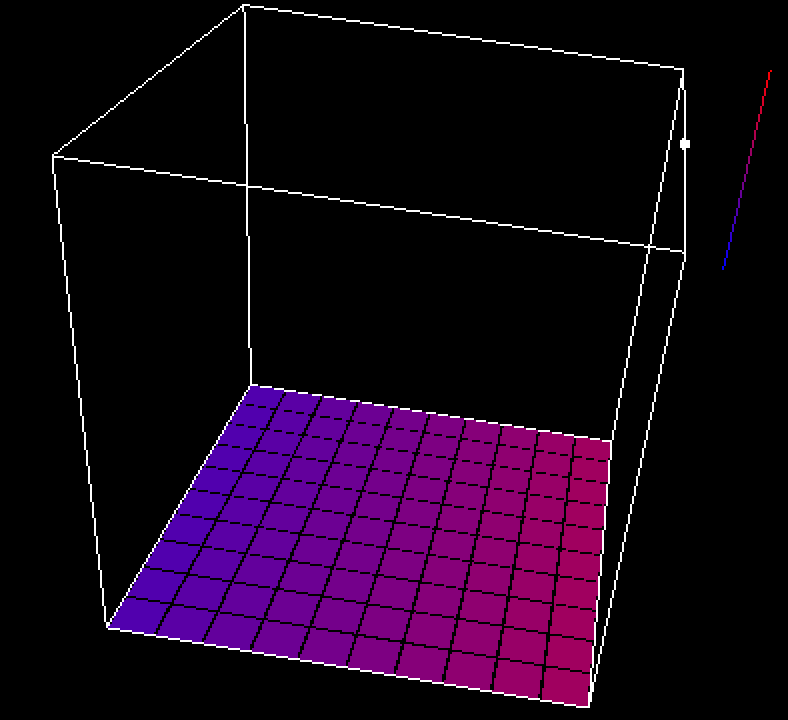}
    \caption{Depiction of forces on the cube/elastic-foundation example described  in~\S\ref{section:elastic-foundation-example}. The small white sphere on the upper right corner of the cube shows the point of application of the external, sinusoidal force. The magenta coloring of the elastic foundation elements indicate that pressure is greater on the right side of the box, as expected (pressure is not concentrated under the point of application since the box is modeled as rigid).}
    \label{fig:elastic-foundation}
\end{figure}

With respect to modeling breadth, Figure~\ref{fig:elastic-foundation} shows that the constraint based approach is able to reasonably simulate compliant contact, including generating cogent pressure distributions. In the next section, we will see that the constraint-based approach can simulate Coulomb friction \emph{without increasing computational hardness} and with minimal creep, which a purely ODE-based approach would struggle to effect efficiently.

\begin{table}[htpb]
\centering
\begin{tabular}{|c|c|l|l|l|}
\hline
Load & Number of elements & Elastic modulus (N/$m^2$) &     Absolute error        & Relative error \\ \hline
5 &    100                & $E = 10^7$     & $2.96 \times 10^{-5}$ & $2.65 \times 10^{-5}$ \\ \hline
5 &    100                & $E = 10^{11}$    & $2.91 \times 10^{-7}$ & $2.60 \times 10^{-7}$ \\ \hline
5 &    900                & $E = 10^7$     & $1.98 \times 10^{-5}$ & $1.77 \times 10^{-5}$ \\ \hline
5 &    900                & $E = 10^{11}$    & $2.71 \times 10^{-9}$ & $2.42 \times 10^{-9}$ \\ \hline
50 &    100                & $E = 10^7$    & $0.813$               & $0.724$ \\ \hline
50 &    100                & $E = 10^{11}$   & $0.814$               & $0.728$ \\\hline
50 &    900                & $E = 10^7$    & $1.0 \times 10^{-2}$  & $8.87 \times 10^{-3}$ \\ \hline 
50 &    900                & $E = 10^{11}$   & $1.0 \times 10^{-2}$ & $8.95 \times 10^{-3}$ \\ \hline 
\end{tabular}
\caption{\label{table:error}Accuracy for the constraint-based approach at $h = 10^{-2}$. Error compared against \texttt{ode45}. \emph{Error improves
significantly for} $h = 10^{-3}$ (not listed above); the absolute errors for 100 elements with $E = 10^7$ N/$m^2$ (fifth row) and $E = 10^11$ N/$m^2$ (sixth row) drop to $3.14 \times 10^{-3}$ and $9.22 \times 10^{-5}$, respectively.}
\end{table}


\section{Modeling Contact with Coulomb friction}
\label{section:modeling-contact}
The path to modeling contact and friction extends the concept described in the previous section. Assume that the unilateral constraint from that section now models two bodies contacting over a surface, and that $\phi_n \to \mathbb{R}$ corresponds to the signed depth of indentation; for sake of ready interpretability and without lack of generality, consider one of the bodies to be rigid and the other to be compliant. The unilateral constraint,
\begin{align}
\hat{M}_n\ddot{\phi}_n + B\dot{\phi}_n + K\phi_n = \hat{f}_n,    
\end{align}
introduces ``n'' subscripts for the variables above to help us distinguish them from other variables later on this section. This constraint prescribes the stiffness and damping in the direction of the surface normal, which results in the integral of the pressure over the surface, and this integral will be denoted $\lambda_n$. The pressure at any point on that surface would then be given by $\frac{\lambda_n}{a}$, where $a$ is the area of the contact surface. Like the elastic foundation example in the previous section, the contact surface can be divided into numbers of discrete elements in order to approximate the pressure distribution with greater accuracy.

But let us return to considering a single spring element representing the surface deformation. Assuming that the contact exists in the context of a 2D system, a second spring---acting \emph{orthogonally} to the first spring---can model Coulomb friction (see Figure~\ref{fig:song-model}); I will denote this constraint $\phi_r$ (and its corresponding force $\lambda_r$). If the contact exists in the context of a 3D system, two such orthogonal springs would be necessary: one for each of two basis directions orthogonal to the first spring. The two (or three) spring-dampers corresponding to each spring-damper element can be imagined as connected in series (also shown in Figure~\ref{fig:song-model}).

\begin{figure}[htpb]
\includegraphics[width=\linewidth]{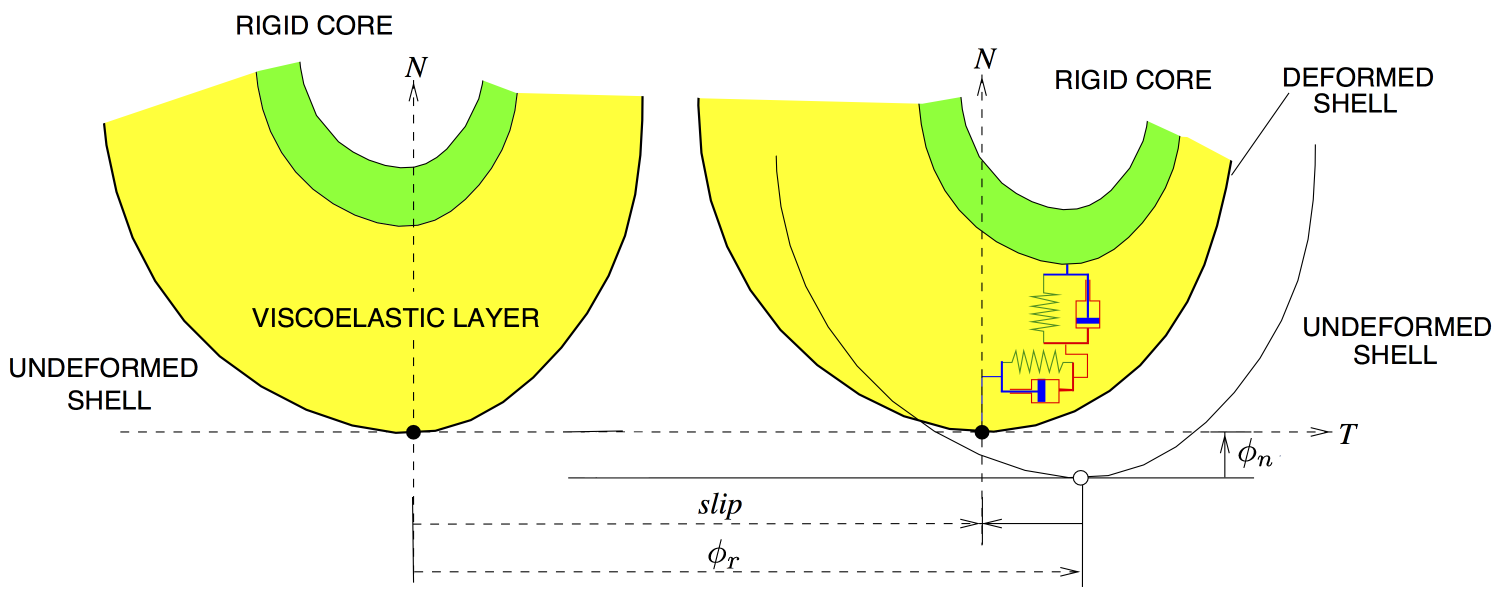}
\caption{A depiction of the model in~\cite{Song:2000a} (illustration adapted from there as well), with the normal and tangential deformations to the surface normal at a presumed point of contact (the case of the deformation, i.e., a contact with another body, is not shown). Note that the virtual springs lie in series.
\label{fig:song-model}}
\end{figure}

Song et al.~\cite{Song:2000a} was the first to propose this particular model of compliant contact, to my knowledge.  To concretize the meaning of the constraints:
\begin{itemize}
\item $\phi_n \leq 0$ represents the negated amount of deformation \emph{along the contact normal} between two bodies at a point of contact. Disjoint (i.e., not touching) bodies do not deform, in which case $\phi = 0$. Lacoursiere noted that $\phi_n$ can also represent a negated potential energy function~\cite{Lacoursiere:2007}.
\item $\dot{\phi_n}$ represents the rate of deformation \emph{along the contact normal} between two bodies at a point of contact. $\dot{\phi} > 0$ indicates that deformation is decreasing. For two \emph{rigid} bodies (i.e., $k = \infty$), $\dot{\phi_n}$ represents the relative velocity at the point of contact projected along the contact normal, where positive $\dot{\phi}$ indicates that two bodies are moving apart at the point of contact.
\item $\phi_r$ represents the amount of signed deformation \emph{orthogonal to the contact normal} between two bodies at a point of contact. 
\item $\dot{\phi_r}$ represents the rate of deformation \emph{orthogonal to the contact normal} between two bodies at a point of contact. For two \emph{rigid} bodies, $\dot{\phi_r}$ represents the relative velocity at the point of contact projected orthogonal to the contact normal. $\dot{\phi}_r = 0$ and $\dot{\phi}_r \neq 0$ correspond to bodies \emph{sticking/rolling} and \emph{sliding}, respectively, at the point of contact.  
\end{itemize}

As was the case in the previous section, we will typically not track (integrate) $\phi_n$, $\dot{\phi}_n$, $\phi_r$, and $\dot{\phi}_r$ directly. $\phi_n$, $\dot{\phi}_n$, and $\dot{\phi}_r$ are often computed using the state $(q, v)$ of the multibody system: for $q$ corresponding to overlapping \emph{undeformed} geometries, a mapping to a kissing, deformed configuration will be computed from which $\phi_n$, $\dot{\phi}_n$, and $\dot{\phi}_r$ can be computed. $\phi_r$ could be tracked by augmenting the system's state variables, but this can be avoided by making the friction model \emph{memoryless} (discussed immediately below). 

\subsection{DPP formulation of a contact constrained system with Coulomb friction, (applicable for $\phi_n \leq 0$)}
I focus on Coulomb friction rather than another friction model because the former is simple  and captures important stick-slip transitions. It also does not incorporate tangential deformation ``memory'', which simplifies the model and its implementation significantly.  In keeping with this simple model of friction, I avoid discriminating between sticking and sliding friction states, which  simplifies both the contact model and its implementation. That lack of discrimination causes the differential equations for the constraint dynamics (the $\phi$ variables) to become numerically stiff (see~\cite{Stewart:2000a}), which does not hamper the efficiency of my implicit approach. The DPP model of a contact constrained with Coulomb friction using my scheme follows:
\begin{align}
\minimize_{\lambda_n, \lambda_r, \phi_n, \dot{\phi}_n, \phi_r, \dot{\phi}_r, \beta_r }\ &  \frac{1}{2}\tr{\lambda_n}\lambda_n + \frac{1}{2}\tr{\lambda_r}\lambda_r + \frac{1}{2}\tr{\beta_r}{\beta_r} \label{eqn:QP-plus-friction-objective}  \\
\dot{q} = &\ Nv \label{eqn:QP-plus-friction-dotq} \\
M\dot{v} = &\ \tr{G_n}\lambda_n + \tr{G_r}\lambda_r + f \label{eqn:QP-plus-friction-dotv} \\
\hat{f}_n = &\ \hat{M}_n\ddot{\phi}_{n} + B\dot{\phi}_n + K\phi_n \label{eqn:dotvarphin} \\
\hat{f}_r = &\ \hat{M}_r\ddot{\phi}_{r} + B\dot{\phi}_r + K\phi_r \label{eqn:dotvarphir} \\
\lambda_n \geq & -K_n\phi_n - B_n\dot{\phi}_n - \hat{M}_n\dot{G}_nv \label{eqn:QP-plus-friction-normal-spring} \\
\lambda_n \geq &\ 0 \label{eqn:QP-plus-friction-compressive-normal} \\
\lambda_r = & -K_r\phi_r - B_r\dot{\phi}_r - \hat{M}_r\dot{G}_rv + \beta_r \label{eqn:QP-plus-friction-tangent-spring} \\
\lambda_r \leq &\ \mu \lambda_n \label{eqn:QP-plus-friction-friction-cone1} \\
-\lambda_r \leq &\ \mu \lambda_n \label{eqn:QP-plus-friction-friction-cone2}.
\end{align}
Similarities between the last QP and~(\ref{eqn:QP-plus-friction-dotq})--(\ref{eqn:dotvarphin}) and~(\ref{eqn:QP-plus-friction-normal-spring})--(\ref{eqn:QP-plus-friction-compressive-normal}) in this QP should be evident. The novel equations are~(\ref{eqn:dotvarphir}) which describes the tangential spring dynamics,~(\ref{eqn:QP-plus-friction-friction-cone1}) and (\ref{eqn:QP-plus-friction-friction-cone2}) which limit the amount of frictional force according to the Coulomb friction model, and~(\ref{eqn:QP-plus-friction-tangent-spring}); the latter will be discussed in detail below. Another difference from the last QP includes separate stiffness/damping terms for normal and tangential directions ($K_n/B_n$ and $K_r/B_r$, respectively).

Neither problem feasibility nor convexity of the above QP will be analyzed here. A time-discretized, more general version (extended with 3D friction constraints, bilateral constraints, and generic unilateral constraints) of the QP-based method introduced immediately below will be analyzed in Appendix~\ref{appendix:minimizing-variables} and found to be feasible and strictly convex. But for now, we examine the presence of the $\beta_r$ term in ~(\ref{eqn:QP-plus-friction-tangent-spring}) by first noting that the nominal constraints for $\lambda_r$, i.e.,
\begin{align}
\lambda_r  = -K_r\phi_r - B_r\dot{\phi}_r - \hat{M}_r\dot{G}_rv \qquad & \textrm{ (spring/damper behavior)} \label{eqn:no-beta} \\
|\lambda_r| \leq \mu \lambda_n \qquad & \textrm{ (Coulomb friction),} 
\end{align}
generally conflict. In response, we enforce one ``hard'' constraint (Coulomb friction) and one ``soft'' one (the spring-damper like behavior) along the contact tangent; the QP finds the frictional forces that respect Coulomb friction yet come as close to the desired damped, oscillatory behavior by introducing the $\beta_r$ term into~(\ref{eqn:no-beta}), yielding ~(\ref{eqn:QP-plus-friction-tangent-spring}). This desired behavior is encouraged by the objective function (Equation~\ref{eqn:QP-plus-friction-objective}), which quadratically penalizes the $\beta_r$ term.

Examining the objective function in detail, notice that each dimension of $\lambda_n$, $\lambda_r$, and $\beta_r$ corresponds to consistent units (e.g., Newtons). Although the units are consistent, it might be conceivable that the solution to the QP increases $\lambda_n$ in order to decrease $\beta_r$. However, the criteria $\tr{\lambda_n}\lambda_n$ and $\tr{\beta_r}\beta_r$ should not be equally valued in the optimization: the former term penalizes deviation from the nonlinear equality constraint $\lambda_n = \max{(0, -K_n\phi_{n} - B_n\dot{\phi}_{n} - \hat{M}_n\dot{G}_nv})$, so a significant weight (penalty) should be assigned to the term $\tr{\lambda_n}\lambda_n$ (I used $1/h$ in my implementation so as to obtain greater accuracy as $h$ is decreased), while keeping the weights of all other terms to unity. This weighting factor will be introduced using the term $\zeta$.

\subsection{The time-discretized DPP (for contact problems in 2D)}
\label{section:formulating-2d-contact problems}
After we discretize in time, we can introduce specific discrete time variables to be determined via optimization. This quadratic program serves as the implicit integration scheme for $\phi_n, \dot{\phi}_n, \phi_r$, and $\dot{\phi}_r$. As in Section~\ref{section:formulating-generic-unilateral-constraints}, the constraint dynamics described in \eqref{eqn:dotvarphin} and \eqref{eqn:dotvarphir} are redundant and need not be considered further. Also as in Section~\ref{section:formulating-generic-unilateral-constraints}: once $\lambda_{n_1}$ and $\lambda_{r_1}$ have been computed, then $v_1$ and $q_1$ are found by evaluating $v_1 = v_0 + h\inv{M}(\tr{G_n}\lambda_{n_1} + \tr{G_r}\lambda_{r_1} + f)$ and $q_1 = q_0 + hNv_1$.
\begin{align}
\minimize_{\lambda_{n_1}, \lambda_{r_1}, \phi_{n_1}, \dot{\phi}_{n_1}, \phi_{r_1}, \dot{\phi}_{r_1}, \beta_r }\ &  \frac{\zeta}{2}\tr{\lambda_{n_1}}\lambda_{n_1} + \frac{1}{2}\tr{\lambda_{r_1}}\lambda_{r_1} + \frac{1}{2}\tr{\beta_r}{\beta_r} \label{eqn:2d-contact-QP-objective-function} \\
\phi_{n_1} = &\ \phi_{n_0} + h \dot{\phi}_{n_1} + O(h^2)\\
\dot{\phi}_{n_1} = &\ \dot{\phi}_{n_0} + h \ddot{\phi}_{n_0}  + O(h^2) \\
\phi_{r_1} = &\ \phi_{r_0} + h \dot{\phi}_{r_1}  + O(h^2) \\
\dot{\phi}_{r_1} = &\ \dot{\phi}_{r_0} + h\ddot{\phi}_{r_0}  + O(h^2) \\
\lambda_{n_1} \geq & -K_n\phi_{n_1} - B_n\dot{\phi}_{n_1} - \hat{M}_{n}\dot{G}_{n}v_0 + O(h) \\
\lambda_{n_1} \geq &\ 0 \\
\lambda_{r_1} = & -K_r\phi_{r_1} - B_r\dot{\phi}_{r_1} - \hat{M}_{r}\dot{G}_{r}v_0 + \beta_r + O(h)\\
\lambda_{r_1} \leq &\ \mu \lambda_{n_1} \\
-\lambda_{r_1} \leq &\ \mu \lambda_{n_1}. \label{eqn:2d-contact-QP-last-constraint-function}
\end{align}
I will show in Appendix~\ref{appendix:minimizing-variables} that the ``free'' problem variables ($\phi_{n_1}, \dot{\phi}_{n_1}, \phi_{r_1}, \dot{\phi}_{r_1}$, and $\beta_r$) can be eliminated, reducing the problem to a purely inequality constrained QP in many fewer variables (two in this case:  $\lambda_{n_1}$ and $\lambda_{r_1}$). Appendix~\ref{appendix:minimizing-variables} also shows that the QP is both feasible and strictly convex.

We will next discuss contact problems in 3D. Before we do, note that the Coulomb friction constraints (e.g., Equations~\ref{eqn:QP-plus-friction-friction-cone1} and \ref{eqn:QP-plus-friction-friction-cone2}) do not specify which direction the force should be applied. If the friction model is truly Coulomb, however, it will be memory-less (implying $K_r = 0$); in that case, $\lambda_{r_1}$ will oppose $\dot{\phi}_{r_1}$, as consistent with the principle of maximum power~\cite{Goyal:1988} from the Coulomb friction model. In other words, setting \mbox{$K_r = 0$} allows the model to faithfully emulate the Coulomb friction model.

\subsection{Three-dimensional contact}
In the three dimensional contact version, we add three more variables per point contact between two bodies---$\phi_{s_1}$, $\dot{\phi}_{s_1}$, and $\lambda_{s_1}$--- which correspond to deformation, deformation time derivative, and applied force for the spring-damper in a second basis direction orthogonal to the contact normal (the spring-damper corresponding to the $r$ index is in another such basis direction). We also need to account for a proper friction ``cone'', as was done using two inequality constraints in~(\ref{eqn:QP-plus-friction-friction-cone1}) and~(\ref{eqn:QP-plus-friction-friction-cone2}); this will be effected using a quadratic constraint for each of the $m$ points of contact. Other than this constraint, the now \emph{quadratically constrained quadratic program} (QCQP) will appear much like the quadratic program used in the DPP from the last section:
\begin{align}
\minimize_{\lambda_{*_1}, \phi_{*_1}, \dot{\phi}_{*_1}, \beta_r, \beta_s } & \frac{\zeta}{2}\tr{\lambda_{n_1}}\lambda_{n_1} + \frac{1}{2}\tr{\lambda_{r_1}}\lambda_{r_1} + \frac{1}{2}\tr{\lambda_{s_1}}\lambda_{s_1} + \ldots \nonumber \\
& \frac{1}{2}\tr{\beta_r}{\beta_r} + \frac{1}{2}\tr{\beta_s}{\beta_s} \\
\dot{q}_1 = &\ Nv_1 + O(h^2) \nonumber \\
M\dot{v}_0 = &\ \tr{G_n}\lambda_{n_1} + \tr{G_r}\lambda_{r_1} + \tr{G}_s\lambda_{s_1} + f + O(h) \nonumber \\
\phi_{n_1} = & \phi_{n_0} + h \dot{\phi}_{n_1} + O(h^2) \\
\dot{\phi}_{n_1} = & \dot{\phi}_{n_0} + h \ddot{\phi}_{n_0} + O(h^2) \\
\phi_{r_1} = & \phi_{r_0} + h \dot{\phi}_{r_1}  + O(h^2) \\
\dot{\phi}_{r_1} = & \dot{\phi}_{r_0} + h\ddot{\phi}_{r_0}   + O(h^2) \\
\phi_{s_1} = & \phi_{s_0} + h \dot{\phi}_{s_1}  + O(h^2) \\
\dot{\phi}_{s_1} = & \dot{\phi}_{s_0} + h\ddot{\phi}_{s_0}  + O(h^2) \\ 
\lambda_{n_1} \geq & -K_n\phi_{n_1} - B_n\dot{\phi}_{n_1} - \hat{M}_{n}\dot{G}_{n}v_0 + O(h) \\
\lambda_{n_1} \geq &\ 0 \\
\lambda_{r_1} = & -K_r\phi_{r_1} - B_r\dot{\phi}_{r_1} - \hat{M}_{r}\dot{G}_{r}v_0 + \beta_r  + O(h) \\
\lambda_{s_1} = & -K_s\phi_{s_1} - B_s\dot{\phi}_{s_1} - \hat{M}_{s}\dot{G}_{s}v_0 + \beta_s  + O(h) \\
\mu_i^2 \lambda_{n_i}^2 & \geq \lambda_{r_i}^2 + \lambda_{s_i}^2 \textrm{ for } i = 1,\ldots,m. \label{eqn:coulomb-3d}
\end{align}
where $\lambda_{*_1}$ represents the variables $\lambda_{n_1}$, $\lambda_{r_1}$ and $\lambda_{s_1}$; $\phi_{*_1}$ represents the variables $\phi_{n_1}$, $\phi_{r_1}$, and $\phi_{s_1}$; and $\dot{\phi}_{*_1}$ represents the variables $\dot{\phi}_{n_1}$, $\dot{\phi}_{r_1}$, and $\dot{\phi}_{s_1}$.

\subsection{Example: box on an inclined plane}
Examples with contact and friction require greater sophistication than the very simple examples I have used prior to now. A good example with contact and friction should demonstrate multiple phenomena: \1 no frictional forces are applied when $\mu = 0$; \2 bodies remain in stiction given sufficiently large $\mu$; \3 all other variables held constant, contacting bodies slide against each other farther with lower $\mu$; and \4 frictional forces should oppose the direction of motion.

To demonstrate these phenomena, I implemented the equations above as a plugin for \texttt{RPI-MATLAB-Simulator} (using \texttt{MATLAB} R2015b) and tested their ``sliding block'' example, which consists of a box of unit mass, width 0.1m, length 0.2m, and height 0.05m that is initially motionless on a ramp inclined $15^{\circ}$. \texttt{RPI-MATLAB-Simulator} models the contact between the block and the ramp using point contacts, which means that the box should be conceived as having four spherical ``feet'' attached to its bottom. To simplify this example, we will model these feet as perfectly rigid and the ramp as an elastic halfspace, which will allow us to use Hertzian contact theory to determine the stiffness ($K_n$) for a given deformation, $d$. Hertzian contact predicts the magnitude of the force from deformation as
\begin{align}
f & = E^* \sqrt{rd^3},
\end{align}
where $E^*$ comprises relative moduli (Young's Modulus and Poisson's Ratio) and $r$ is the radius of the spherical feet. For this example problem, let us assume a radius of 0.01m. We now have the information with which to determine $\phi_n$ and $K_n$. Specifically, we will define the diagonal entries of $K_n$ as
\begin{align}
K_{n_{ii}} & = E^* \sqrt{r},
 \end{align}
where we will assume that the elastic halfspace is comprised of a material with the rigidity of steel (recall that the feet are rigid), so $E^*$ will be $100$ GPa = $10^{11}$ $\textrm{N/m}^2$. $\sqrt{r}$ will be $\frac{1}{10}$ $\textrm{m}^{1/2}$, yielding a pseudo-stiffness of $10^{10}$ $\textrm{N/m}^{1/2}$ (it is only a pseudo-stiffness because it is not quite in units of N/m). That is okay because we will define $\phi_{n_i}$ ($i$ refers to the index of one of the spherical feet) to return units of $\textrm{m}^{3/2}$ rather than m:
\begin{align}
\phi_{n_i} & = d_i^{3/2}.
\end{align}
\texttt{RPI-MATLAB-Simulator} provides the deformation depth, $d_i$ for each contacting sphere/halfspace combination  (it does not compute deformation for the block feet that are not in contact). Selecting damping parameters is never so straightforward; even for highly accurate models, damping must usually be set subjectively to realize qualitatively desired behavior. I set the diagonal entries of $B_n$ to unity as a guess to effect ``a little damping''. I set the diagonal entries of $K_{r}$ and $K_s$ to zero to emulate the Coulomb friction model, which is memoryless as I noted earlier. 

It should be clear that $B_r$ and $B_s$ (called $B$ in the remainder of this paragraph) should be set as large as possible. We want the friction coefficient to limit the force applied, not the dynamics of the virtual damper. Unfortunately, the numerical range of the optimization problem increases with $B$: the number of digits of precision required is $O(hB)$, which can challenge the default tolerances used by some  numerical libraries. I found that setting $B = \frac{10^6}{h}$ meters results in no creep and no failed solves.    
 
In the context of this example, I varied friction using $\mu = \{ 0, 0.125, 0.25, 0.375 \}$. I guaranteed that there were no frictional forces applied when $\mu = 0$, and we observed the box remain in stiction at $\mu = 0.375$. Between those values, the box slid for distances inversely proportional to the coefficient of friction, as Figure~\ref{fig:sliding-box} shows. I arbitrarily selected a step size of 0.01s for simulating the scenario.   

The normal component of the contact forces tends to lie around $2.5$N per point in this example, which is consistent with the gravitational force of $9.8$N acting on this block. From Hertzian contact theory, this amount of force indicates that the deformation should be around $10^{-8}$m, which seems reasonable given the load and the material. 

\begin{figure}[htbp]
\centering
\includegraphics[width=\linewidth]{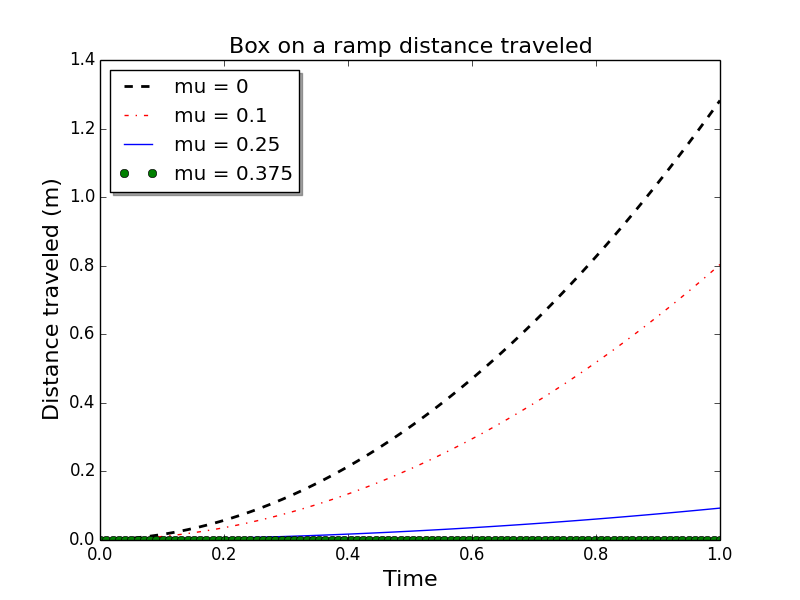}
\caption{The distance that the box on the inclined ramp travels over a second for various coefficients of friction. }
\label{fig:sliding-box}
\end{figure}

\subsection{Example: stack of spheres}
I used the example of stacks of spheres to assess how well the solver scales computationally with numbers of  contact constraints. Three constraints (corresponding to $\phi_n, \phi_r, \phi_s$) were used between each pair of contacting spheres. Although a large number of contacting spheres will generate a large number of optimization variables, the problem is inherently sparse: two contacting spheres yield only a $3 \times 3$ dense submatrix in the problem Hessian matrix. If each sphere touches at most two others, the ceiling on the total number of nonzero Hessian entries will be $18m$, where $m$ is the number of spheres, while the Hessian matrix is expected to be $3m \times 3m$, i.e., $9m^2$ total entries.

The simulation was started with all spheres in mid-air, spaced evenly apart vertically (all spheres are in perfect horizontal alignment), and at rest. As the simulation progresses, the bottom sphere collides with the ground, the second sphere collides with the bottom sphere, etc. Wave propagation effects moving up the stack and due to the various collisions are clearly visible in the accompanying video. The stability of the approach is such that the spheres are able to form a stable stack without any constraints (e.g., walls around the stack) keeping them balanced.

I simulated the stack of spheres using \texttt{RPI-MATLAB-Simulator} for 500 iterations at a step size of 10 ms (selected arbitrarily). I tested stacks of up to 100 spheres, each of which used Young's Modulus of $E = 10^{11} \textrm{N}/\textrm{m}^2$ and damping parameter $B = I\frac{\textrm{N}\cdot \textrm{s}}{\textrm{m}}$. Performance results are plotted in Figure~\ref{fig:stack-performance}. Although explicit schemes should be capable of simulating stacks of softer spheres much more quickly, we are not aware of a competing approach that can robustly simulate stacks of such stiff spheres nearly as fast (cf.~\cite{Lacoursiere:2003}, which describes numerous issues with simulating stacks of rigid objects quickly and robustly). Thus, we trade peak performance at simulating contacting rigid bodies over a subset of material properties for good overall performance on arbitrary materials.  

\begin{figure}[htpb]
\includegraphics[width=\linewidth]{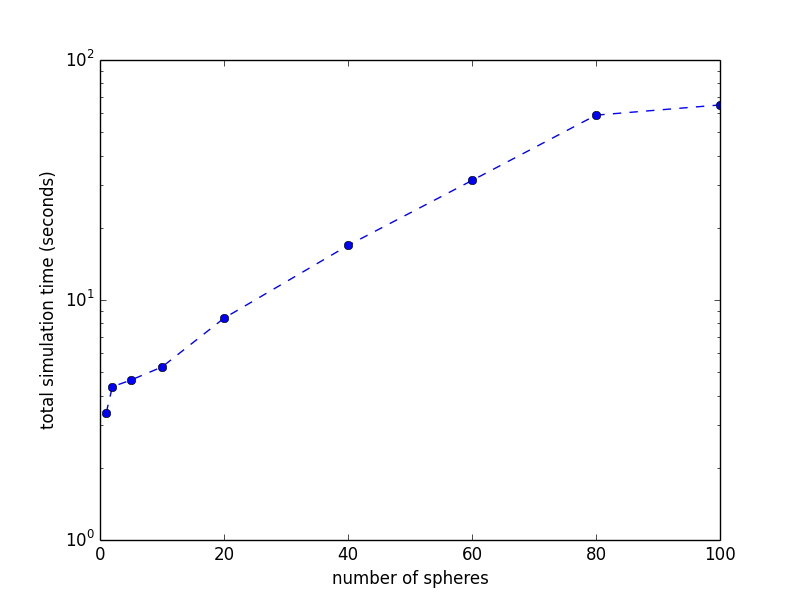}
\caption{Times (vertical axis) required to simulate stacks of numbers of spheres (horizontal axis) for one second of virtual time. Each sphere in the stack adds three variables to the convex optimization problem.
\label{fig:stack-performance}}
\end{figure}

\section{Discussion}
\label{section:discussion}
My approach is first-order accurate, assuming that all events are isolated. For simulation, the approach should be competitive (if not much faster) and more numerically robust than approaches that model constraints as rigid: convex optimization problems are far easier to solve than complementarity problems. My method should be faster on many problems than Newton-Raphson-based implicit integrators for simulations than model constraints as springs and dampers; the exception will be those problems that have many more constraints than state variables (like the example in \S\ref{section:unilateral-msd}, for which my method still exhibited the best worst-case performance compared to state of the art explicit and implicit integrators). And my method has the advantage of being very easy to implement, despite the appearances from the pages of algebra that produced the final simplifications.

This article did not explore applications other than simulation (i.e., control, state estimation, and system identification). Compliant models used in those applications all suffer from the drawback that deformation can be very hard, if not infeasible, to measure particularly when materials are stiff and deformations small. My constraint-based approach seems well suited for these applications because it can estimate forces in the absence of deformation measurements (by treating the deformation as zero) and because it does not suffer from the problems of rigid models. Ultimately, I expect this kind of quasistatic treatment to be reasonable when the constraint acts more plastic than elastic (e.g., like when an object is contacting mud) or when the constraint is very stiff and vibrations dissipate rapidly. I plan to investigate the reasonableness of this quasistatic treatment in future work. 

I also plan future investigations that attempt to exploit this constraint solver to do faster and more accurate state estimation, control, and system identification in applications where robots interact with their environments through contact.   

\section*{Acknowledgements}
I thank Dylan Shell, Paul Mitiguy, Hongkai Dai, Ryan Elandt, and Andy Ruina for providing excellent feedback on the manuscript and Michael Sherman for posing early critiques that helped shape its direction.

\appendix

\section{Minimizing variables for all constraints (bilateral, unilateral, and contact constraints in 3D)}
\label{appendix:minimizing-variables}

This section considers how to minimize all variables for systems with bilateral constraints (variables with subscript $_b$), unilateral constraints (variables with subscript $_u$), and contact constraints in 3D (variables with subscripts $_n$, $_r$, or $_s$).   I have included this section to keep the reader from having to do these simplifications. We first formulate the necessary differential optimization problem as a straightforward extension of the techniques presented in Sections~\ref{section:formulating-bilateral-constraints}
 and~\ref{section:formulating-generic-unilateral-constraints}:
\begin{align}
\minimize_{\lambda, \phi_1, \dot{\phi}_1, \beta}\ & \frac{1}{2}(\zeta \tr{\lambda_{n_1}}\lambda_{n_1} + \zeta \tr{\lambda_{u_1}}{\lambda_{u_1}} + \tr{\lambda_{r_1}}\lambda_{r_1} + \tr{\beta_r}{\beta_r} + \tr{\lambda_{s_1}}\lambda_{s_1} + \tr{\beta_s}\beta_s) \\
\phi_{u_1} & = \phi_{u_0} + h \dot{\phi}_{u_1} + O(h^2)\\
\dot{\phi}_{u_1} & = \dot{\phi}_{u_0} + h \ddot{\phi}_{u_0} + O(h^2)\\
\phi_{b_1} & = \phi_{b_0} + h \dot{\phi}_{b_1} + O(h^2)\\
\dot{\phi}_{b_1} & = \dot{\phi}_{b_0} + h \ddot{\phi}_{b_0} + O(h^2)\\
\phi_{n_1} & = \phi_{n_0} + h \dot{\phi}_{n_1} + O(h^2)\\
\dot{\phi}_{n_1} & = \dot{\phi}_{n_0} + h \ddot{\phi}_{n_0} + O(h^2)\\
\phi_{r_1} & = \phi_{r_0} + h \dot{\phi}_{r_1} + O(h^2)\\
\dot{\phi}_{r_1} & = \dot{\phi}_{r_0} + h 
\ddot{\phi}_{r_0} + O(h^2) \\
\phi_{s_1} & = \phi_{s_0} + h \dot{\phi}_{s_1} + O(h^2)\\
\dot{\phi}_{s_1} & = \dot{\phi}_{s_0} + h 
\ddot{\phi}_{s_0} + O(h^2)\\
\lambda_{u_1} & \geq - K_u\phi_{u_1} - B_u\dot{\phi}_{u_1} - \hat{M}_{u}\dot{G}_{u}v_0 + O(h)\\
\lambda_{u_1} & \geq 0 \\
\lambda_{n_1} & \geq - K_n\phi_{n_1} - B_n\dot{\phi}_{n_1} - \hat{M}_{n}\dot{G}_{n}v_0 + O(h) \\
\lambda_{n_1} & \geq 0 \\
\lambda_{r_1} & = -K_r\phi_{r_1} - B_r\dot{\phi}_{r_1} - \hat{M}_{r}\dot{G}_{r}v_0 + \beta_r + O(h) \\
\lambda_{s_1} & = -K_s\phi_{s_1} - B_s\dot{\phi}_{s_1} - \hat{M}_{s}\dot{G}_{s}v_0 + \beta_s + O(h) \\
\lambda_{b_1} & = -K_b\phi_{b_1} - B_b\dot{\phi}_{b_1} - \hat{M}_{b}\dot{G}_{b}v_0 + O(h)\\
\mu^2 \lambda_{n_1}^2 & \geq \lambda_{r_1}^2 + \lambda_{s_1}^2.
\end{align}

As with the previous section, this section will minimize the number of variables, equation numbers will only be provided when an equation changes, but the whole QP will be reproduced at each step. We first replace all of the $\phi_1$ terms: $\phi_{b_1}, \phi_{u_1}, \phi_{n_1}, \phi_{r_1}, \textrm{ and } \phi_{s_1}$:
\begin{align}
\minimize_{\lambda, \dot{\phi}_1, \beta}\ & \frac{1}{2}(\zeta \tr{\lambda_{n_1}}\lambda_{n_1} + \zeta \tr{\lambda_{u_1}}{\lambda_{u_1}} + \tr{\lambda_{r_1}}\lambda_{r_1} + \tr{\beta_r}{\beta_r} + \tr{\lambda_{s_1}}\lambda_{s_1} + \tr{\beta_s}\beta_s) \nonumber \\
\dot{\phi}_{u_1} & = \dot{\phi}_{u_0} + h \ddot{\phi}_{u} + O(h^2) \nonumber \\
\dot{\phi}_{b_1} & = \dot{\phi}_{b_0} + h \ddot{\phi}_{b} + O(h^2) \nonumber \\
\dot{\phi}_{n_1} & = \dot{\phi}_{n_0} + h \ddot{\phi}_{n} + O(h^2) \nonumber \\
\dot{\phi}_{r_1} & = \dot{\phi}_{r_0} + h 
\ddot{\phi}_{r_0} + O(h^2) \nonumber \\
\dot{\phi}_{s_1} & = \dot{\phi}_{s_0} + h \ddot{\phi}_{s_0} + O(h^2) \nonumber \\
\lambda_{u_1} & \geq - K_u\phi_{u_0} - (hK_u + B_u)\dot{\phi}_{u_1} - \hat{M}_{u}\dot{G}_{u}v_0 + O(h) \\
\lambda_{u_1} & \geq 0 \nonumber \\
\lambda_{n_1} & \geq - K_n\phi_{n_0} - (hK_n + B_n)\dot{\phi}_{n_1} - \hat{M}_{n}\dot{G}_{n}v_0 + O(h) \\
\lambda_{n_1} & \geq 0 \nonumber \\
\lambda_{r_1} & = -K_r\phi_{r_0} - (hK_r + B_r)\dot{\phi}_{r_1} - \hat{M}_{r}\dot{G}_{r}v_0 + \beta_r + O(h) \\
\lambda_{s_1} & = -K_r\phi_{s_0} - (hK_s + B_s)\dot{\phi}_{s_1} - \hat{M}_{s}\dot{G}_{s}v_0 + \beta_s + O(h) \\
\lambda_{b_1} & = -K_b\phi_{b_0} - (hK_b + B_b)\dot{\phi}_{b_1} - \hat{M}_{b}\dot{G}_{b}v_0 + O(h) \\
\mu^2 \lambda_{n_1}^2 & \geq \lambda_{r_1}^2 + \lambda_{s_1}^2. \nonumber
\end{align}

We then replace all of the $\dot{\phi}_1$ terms: $\dot{\phi}_{b_1}, \dot{\phi}_{u_1}, \dot{\phi}_{n_1}, \dot{\phi}_{r_1}, \textrm{ and } \dot{\phi}_{s_1}$:
\begin{align}
\minimize_{\lambda, \dot{\phi}_1, \beta}\ & \frac{1}{2}(\zeta \tr{\lambda_{n_1}}\lambda_{n_1} + \zeta \tr{\lambda_{u_1}}{\lambda_{u_1}} + \tr{\lambda_{r_1}}\lambda_{r_1} + \tr{\beta_r}{\beta_r} + \tr{\lambda_{s_1}}\lambda_{s_1} + \tr{\beta_s}\beta_s) \nonumber \\
\lambda_{u_1} & \geq - K_u\phi_{u_0} - (hK_u + B_u)\underbrace{(\dot{\phi}_{u_0} + h \ddot{\phi}_{u} + O(h^2))}_{\dot{\phi}_{u_1}} - \ldots \nonumber \\
& \qquad \hat{M}_{u}\dot{G}_{u}v_0 + O(h) \\
\lambda_{u_1} & \geq 0 \nonumber \\
\lambda_{n_1} & \geq - K_n\phi_{n_0} - (hK_n + B_n)\underbrace{(\dot{\phi}_{n_0} + h \ddot{\phi}_{n} + O(h^2))}_{\dot{\phi}_{n_1}} - \ldots \nonumber \\
& \qquad \hat{M}_{n}\dot{G}_{n}v_0 + O(h) \\
\lambda_{n_1} & \geq 0 \nonumber \\
\lambda_{r_1} & = -K_r\phi_{r_0} - (hK_r + B_r)\underbrace{(\dot{\phi}_{r_0} + h 
\ddot{\phi}_{r_0} + O(h^2))}_{\dot{\phi}_{r_1}} - \ldots \nonumber \\
& \qquad \hat{M}_{r}\dot{G}_{r}v_0 + \beta_r + O(h) \\
\lambda_{s_1} & = -K_r\phi_{s_0} - (hK_s + B_s)\underbrace{(\dot{\phi}_{s_0} + h 
\ddot{\phi}_{s_0} + O(h^2))}_{\dot{\phi}_{s_1}} - \ldots \nonumber \\
& \qquad \hat{M}_{s}\dot{G}_{s}v_0 + \beta_s + O(h) \\
\lambda_{b_1} & = -K_b\phi_{b_0} - (hK_b + B_b)\underbrace{(\dot{\phi}_{b_0} + h \ddot{\phi}_{b} + O(h^2))}_{\dot{\phi}_{b_1}} - \ldots \nonumber \\
& \qquad \hat{M}_{b}\dot{G}_{b}v_0 + O(h) \\
\mu^2 \lambda_{n_1}^2 & \geq \lambda_{r_1}^2 + \lambda_{s_1}^2. \nonumber
\end{align}

As demonstrated many times previously, the $hO(h^2)$ terms are subsumed by the $O(h)$ terms, so we make that simplification and also leverage the identities $\dot{v}_0 = \inv{M}(\tr{G_n}\lambda_{n_1} + \tr{G_r}\lambda_{r_1} + \tr{G_s}\lambda_{s_1} + \tr{G_u}\lambda_{u_1} + \tr{G_b}\lambda_{b_1} + f) + O(h)$, $\ddot{\phi}_{n_0} = G_n\dot{v}_0 + \dot{G}_nv_0$, $\ddot{\phi}_{r_0} = G_r\dot{v}_0 + \dot{G}_rv_0$, $\ddot{\phi}_{s_0} = G_s\dot{v}_0 + \dot{G}_sv_0$, $\ddot{\phi}_{u_0} = G_u\dot{v}_0 + \dot{G}_uv_0$, $\ddot{\phi}_{u_0} = G_u\dot{v}_0 + \dot{G}_uv_0$, and $\ddot{\phi}_{b_0} = G_b\dot{v}_0 + \dot{G}_bv_0$:
\begin{align}
\minimize_{\lambda, \beta}\ & \frac{1}{2}(\zeta \tr{\lambda_{n_1}}\lambda_{n_1} + \zeta \tr{\lambda_{u_1}}{\lambda_{u_1}} + \tr{\lambda_{r_1}}\lambda_{r_1} + \tr{\beta_r}{\beta_r}  + \tr{\beta_s}{\beta_s}) \nonumber \\
\lambda_{u_1} & \geq - \hat{M}_{u}\dot{G}_{u}v_0 - K_u\phi_{u_0} - (hK_u + B_u)(\dot{\phi}_{u_0} + h \ldots \nonumber \\ & \underbrace{(G_u\inv{M}(\tr{G_n}\lambda_{n_1} + \tr{G_r}\lambda_{r_1} + \tr{G_s}\lambda_{s_1} + \tr{G_u}\lambda_{u_1} + \tr{G_b}\lambda_{b_1} + f) + \dot{G}_uv_0)}_{\ddot{\phi}_{u_0}}) \\
\lambda_{u_1} & \geq 0 \nonumber \\
\lambda_{n_1} & \geq - \hat{M}_{n}\dot{G}_{n}v_0 - K_n\phi_{n_0} - (hK_n + B_n)(\dot{\phi}_{n_0} + h \ldots \nonumber \\
& \underbrace{(G_n\inv{M}(\tr{G_n}\lambda_{n_1} + \tr{G_r}\lambda_{r_1} + \tr{G_s}\lambda_{s_1} + \tr{G_u}\lambda_{u_1} + \tr{G_b}\lambda_{b_1} + f) + \dot{G}_nv_0)}_{\ddot{\phi}_{n_0}}) \\
\lambda_{n_1} & \geq 0 \nonumber \\
\lambda_{r_1} & = \beta_r - \hat{M}_{r}\dot{G}_{r}v_0 - K_r\phi_{r_0} - (hK_r + B_r)(\dot{\phi}_{r_0} + h \ldots \nonumber \\
& \underbrace{(G_r\inv{M}(\tr{G_n}\lambda_{n_1} + \tr{G_r}\lambda_{r_1} + \tr{G_s}\lambda_{s_1} + \tr{G_u}\lambda_{u_1} + \tr{G_b}\lambda_{b_1} + f) + \dot{G}_rv_0)}_{\ddot{\phi}_{r_0}}) \\
\lambda_{s_1} & = \beta_s - \hat{M}_{s}\dot{G}_{s}v_0 -K_s\phi_{s_0} - (hK_s + B_s)(\dot{\phi}_{s_0} + h \ldots \nonumber \\
& \underbrace{(G_s\inv{M}(\tr{G_n}\lambda_{n_1} + \tr{G_r}\lambda_{r_1} + \tr{G_s}\lambda_{s_1} + \tr{G_u}\lambda_{u_1} + \tr{G_b}\lambda_{b_1} + f) + \dot{G}_rv_0)}_{\ddot{\phi}_{s_0}}) \\
\lambda_{b_1} & = - \hat{M}_{b}\dot{G}_{b}v_0 -K_b\phi_{b_0} - (hK_b + B_b)(\dot{\phi}_{b_0} + h \ldots \nonumber \\
& \underbrace{(G_b\inv{M}(\tr{G_n}\lambda_{n_1} + \tr{G_r}\lambda_{r_1} + \tr{G_s}\lambda_{s_1} + \tr{G_u}\lambda_{u_1} + \tr{G_b}\lambda_{b_1} + f) + \dot{G}_bv_0)}_{\ddot{\phi}_{b_0}}) \label{eqn:lambda-b-def} \\
\mu^2 \lambda_{n_1}^2 & \geq \lambda_{r_1}^2 + \lambda_{s_1}^2. \nonumber
\end{align}

Next, we replace the $\lambda_{b_1}$ term by first using the identity $\hat{M}_b \equiv \inv{(G_b\inv{M}\tr{G_b})}$ and then rearranging~(\ref{eqn:lambda-b-def}) to:
\begin{align}
(I + (h^2K_b + hB_b)\inv{\hat{M}_{b_0}})\lambda_{b_1} = &\ - \hat{M}_{u}\dot{G}_{u}v_0 - K_b\phi_{b_0} - (hK_b + B_b)(\ldots \nonumber \\
& \qquad \dot{\phi}_{b_0} + h (G_b\inv{M}(\tr{G_n}\lambda_{n_1} + \tr{G_r}\lambda_{r_1} + \ldots \nonumber \\
& \qquad \tr{G_s}\lambda_{s_1} + \tr{G_u}\lambda_{u_1} + f) + \ldots \nonumber \\
& \qquad \dot{G}_bv_0)) + O(h) \label{eqn:lambda-b-def2},
\end{align}
then defining
\begin{align}
c \equiv &\ -\hat{M}_u\dot{G}_uv_0 -K_b\phi_{b_0} - (hK_b + B_b)(\dot{\phi}_{b_0} + h (G_b\inv{M}f + \dot{G}_bv_0) \\
D \equiv &\ -(h^2K_b + hB_b)G_b\inv{M}\\
C \equiv &\ (I - D\tr{G_b})\\
G^* \equiv &\ \tr{\begin{bmatrix}\tr{G_n} & \tr{G_r} & \tr{G_s} & \tr{G_u} \end{bmatrix}}\\
\lambda^*_1 \equiv &\ \tr{\begin{bmatrix} \tr{\lambda_{n_1}} & \tr{\lambda_{r_1}} & \tr{\lambda_{s_1}} & \tr{\lambda_{u_1}} \end{bmatrix}},
\end{align}
allowing us to reformulate~(\ref{eqn:lambda-b-def2}) into $C\lambda_{b_1} = c + D\tr{G^*}\lambda^*_1 + O(h)$ or, equivalently, $\lambda_{b_1} = \inv{C}(c + D\tr{G^*}\lambda^*_1) + O(h)$. Using this latter relationship, we can replace $\lambda_{b_1}$:
\begin{align}
\minimize_{\lambda, \beta}\ & \frac{1}{2}(\zeta \tr{\lambda_{n_1}}\lambda_{n_1} + \zeta \tr{\lambda_{u_1}}{\lambda_{u_1}} + \tr{\lambda_{r_1}}\lambda_{r_1} + \tr{\beta_r}{\beta_r} + \tr{\beta_s}\beta_s) \nonumber \\
\lambda_{u_1} & \geq -\hat{M}_u\dot{G}_uv_0 - K_u\phi_{u_0} - (hK_u + B_u)(\dot{\phi}_{u_0} + h (G_u\inv{M}(\tr{G_n}\lambda_{n_1} + \ldots \nonumber \\
& \quad \tr{G_r}\lambda_{r_1} + \tr{G_s}\lambda_{s_1} + \tr{G_u}\lambda_{u_1} + \tr{G_b}\underbrace{\inv{C}(c + D\tr{G^*}\lambda^*_1 + O(h))}_{\lambda_{b_1}} + f) + \ldots \nonumber \\
& \quad \dot{G}_uv_0)) + O(h)\\
\lambda_{u_1} & \geq 0 \nonumber \\
\lambda_{n_1} & \geq -\hat{M}_n\dot{G}_nv_0 - K_n\phi_{n_0} - (hK_n + B_n)(\dot{\phi}_{n_0} + h (G_n\inv{M}(\tr{G_n}\lambda_{n_1} + \ldots \nonumber \\
& \quad \tr{G_r}\lambda_{r_1} + \tr{G_s}\lambda_{s_1} + \tr{G_u}\lambda_{u_1} + \tr{G_b}\underbrace{\inv{C}(c + D\tr{G^*}\lambda^*_1  + O(h))}_{\lambda_{b_1}} + f) + \ldots \nonumber \\
& \quad \dot{G}_nv_0)) + O(h) \\
\lambda_{n_1} & \geq 0 \nonumber \\
\lambda_{r_1} & = -\hat{M}_r\dot{G}_rv_0 -K_r\phi_{r_0} - (hK_r + B_r)(\dot{\phi}_{r_0} + h (G_r\inv{M}(\tr{G_n}\lambda_{n_1} + \ldots \nonumber \\
& \quad  \tr{G_r}\lambda_{r_1} + \tr{G_s}\lambda_{s_1} + \tr{G_u}\lambda_{u_1} + \tr{G_b}\underbrace{\inv{C}(c + D\tr{G^*}\lambda^*_1  + O(h))}_{\lambda_{b_1}} + f) + \ldots \nonumber \\
& \quad \dot{G}_rv_0)) + \beta_r + O(h) \\
\lambda_{s_1} & = -\hat{M}_s\dot{G}_sv_0 -K_s\phi_{s_0} - (hK_s + B_s)(\dot{\phi}_{s_0} + h (G_s\inv{M}(\tr{G_n}\lambda_{n_1} + \ldots \nonumber \\
& \quad  \tr{G_r}\lambda_{r_1} + \tr{G_s}\lambda_{s_1} + \tr{G_u}\lambda_{u_1} + \tr{G_b}\underbrace{\inv{C}(c + D\tr{G^*}\lambda^*_1 + O(h))}_{\lambda_{b_1}} + f) + \ldots \nonumber \\
& \quad \dot{G}_sv_0)) + \beta_s + O(h) \\
\mu^2 \lambda_{n_1}^2 & \geq \lambda_{r_1}^2 + \lambda_{s_1}^2. \nonumber
\end{align}
We will now remove the subsumed $hO(h)$ terms and the $\lambda^*_1$ terms, replacing the latter with the expansion $\lambda^*_1 = \tr{\begin{bmatrix} \tr{\lambda_{n_1}} & \tr{\lambda_{r_1}} & \tr{\lambda_{s_1}} & \tr{\lambda_{u_1}}\end{bmatrix}}$ temporarily (they will be reintroduced for compactness later).
\begin{align}
\minimize_{\lambda, \beta}\ & \frac{1}{2}(\zeta \tr{\lambda_{n_1}}\lambda_{n_1} + \zeta \tr{\lambda_{u_1}}{\lambda_{u_1}} + \tr{\lambda_{r_1}}\lambda_{r_1} + \tr{\beta_r}{\beta_r} + \tr{\beta_s}\beta_s) \nonumber \\
\lambda_{u_1} & \geq -\hat{M}_u\dot{G}_uv_0 - K_u\phi_{u_0} - (hK_u + B_u)(\dot{\phi}_{u_0} + h (G_u\inv{M}(\tr{G_b}\inv{C}D + I)(\ldots \nonumber \\
& \qquad \tr{G_n}\lambda_{n_1} + \tr{G_r}\lambda_{r_1} + \tr{G_s}\lambda_{s_1} + \tr{G_u}\lambda_{u_1} + \tr{G_b}\inv{C}c + f) + \ldots \nonumber \\
& \qquad \dot{G}_uv_0)) + O(h) \\
\lambda_{u_1} & \geq 0 \nonumber \\
\lambda_{n_1} & \geq -\hat{M}_n\dot{G}_nv_0 - K_n\phi_{n_0} - (hK_n + B_n)(\dot{\phi}_{n_0} + h (G_n\inv{M}(\tr{G_b}\inv{C}D + I)(\ldots \nonumber \\
& \qquad \tr{G_n}\lambda_{n_1} + \tr{G_r}\lambda_{r_1} + \tr{G_s}\lambda_{s_1} + \tr{G_u}\lambda_{u_1} + \tr{G_b}\inv{C}c + f) + \ldots \nonumber \\
& \qquad \dot{G}_nv_0)) + O(h) \\
\lambda_{n_1} & \geq 0 \nonumber \\
\lambda_{r_1} & = -\hat{M}_r\dot{G}_rv_0 -K_r\phi_{r_0} - (hK_r + B_r)(\dot{\phi}_{r_0} + h (G_r\inv{M}(\tr{G_b}\inv{C}D + I)(\ldots \nonumber \\
& \qquad \tr{G_n}\lambda_{n_1} + \tr{G_r}\lambda_{r_1} +  \tr{G_s}\lambda_{s_1} + \tr{G_u}\lambda_{u_1} + \tr{G_b}\inv{C}c + f) + \ldots \nonumber \\
& \qquad \dot{G}_rv_0)) + \beta_r + O(h) \\
\lambda_{s_1} & = -\hat{M}_s\dot{G}_sv_0 -K_s\phi_{s_0} - (hK_s + B_s)(\dot{\phi}_{s_0} + h (G_s\inv{M}(\tr{G_b}\inv{C}D + I)(\ldots \nonumber \\
& \qquad \tr{G_n}\lambda_{n_1} + \tr{G_r}\lambda_{r_1} +  \tr{G_s}\lambda_{s_1} + \tr{G_u}\lambda_{u_1} + \tr{G_b}\inv{C}c + f) + \ldots \nonumber \\
& \qquad \dot{G}_sv_0)) + \beta_s + O(h) \\
\mu^2 \lambda_{n_1}^2 & \geq \lambda_{r_1}^2 + \lambda_{s_1}^2 \nonumber
\end{align}

Finally, we replace the $\beta_r$ and $\beta_s$ terms with
\begin{align}
\beta_r & = \lambda_{r_1} + \hat{M}_r\dot{G}_rv_0 + K_r\phi_{r_0} + (hK_r + B_r)(\dot{\phi}_{r_0} + h (G_r\inv{M}(\tr{G_b}\inv{C}D + I)(\ldots \nonumber \\
& \quad \tr{G_n}\lambda_{n_1} + \tr{G_r}\lambda_{r_1} + \tr{G_s}\lambda_{s_1} + \tr{G_u}\lambda_{u_1} + \tr{G_b}\inv{C}c + f) + \dot{G}_rv_0)) + O(h) \nonumber \\
\beta_s & = \lambda_{s_1} + \hat{M}_s\dot{G}_sv_0 + K_s\phi_{s_0} + (hK_s + B_s)(\dot{\phi}_{s_0} + h (G_s\inv{M}(\tr{G_b}\inv{C}D + I)(\ldots \nonumber \\
& \quad \tr{G_n}\lambda_{n_1} + \tr{G_r}\lambda_{r_1} + \tr{G_s}\lambda_{s_1} + \tr{G_u}\lambda_{u_1} + \tr{G_b}\inv{C}c + f) + \dot{G}_sv_0)) + O(h), \nonumber
\end{align}
and we remove the $\lambda_{r_1}$ and $\lambda_{s_1}$ constraint equations from the convex program as well. We will simplify this statement by grouping $\lambda$ terms and constant terms to make $\beta_r = \lambda_{r_1} + E\lambda^*_1 + d + O(h)$ and $\beta_s = \lambda_{s_1} + F\lambda^*_1 + e + O(h)$, where:
\begin{align}
E \equiv &\ (h^2K_r + hB_r)G_r\inv{M}(\tr{G_b}\inv{C}D + I)\tr{G^*}, \\
F \equiv &\ (h^2K_s + hB_s)G_s\inv{M}(\tr{G_b}\inv{C}D + I)\tr{G^*}, \\
d \equiv &\ \hat{M}_r\dot{G}_rv_0 + K_r\phi_{r_0} + (hK_r + B_r)(\dot{\phi}_{r_0} + h(G_r\inv{M}f + \dot{G}_rv_0)), \\
e \equiv &\ \hat{M}_s\dot{G}_sv_0 + K_s\phi_{s_0} + (hK_s + B_s)(\dot{\phi}_{s_0} + h(G_s\inv{M}f + \dot{G}_sv_0)).
\end{align}
The simplified optimization problem, after removing asymptotic error terms, is now:
\begin{align}
\minimize_{\lambda}\ & \frac{1}{2}(\zeta \tr{\lambda_{n_1}}\lambda_{n_1} + \zeta \tr{\lambda_{u_1}}{\lambda_{u_1}} + \tr{\lambda_{r_1}}(2\lambda_{r_1} + 2E\lambda^*_1 + 2d) +  \tr{\lambda_{s_1}}(2\lambda_{s_1} + \ldots \nonumber \\
& \quad 2F\lambda^*_1 + 2e) + \tr{\lambda^*_1}\tr{E}(E\lambda^*_1 + 2d) + \tr{\lambda^*_1}\tr{F}(F\lambda^*_1 + 2e)) \\
\lambda_{u_1} & \geq - \hat{M}_u\dot{G}_uv_0 - K_u\phi_{u_0} - (hK_u + B_u)(\dot{\phi}_{u_0} + h (G_u\inv{M}((\tr{G_b}\inv{C}D + I)(\ldots \nonumber \\
& \qquad \tr{G_n}\lambda_{n_1} + \tr{G_r}\lambda_{r_1} + \tr{G_s}\lambda_{s_1} + \tr{G_u}\lambda_{u_1}) + \tr{G_b}\inv{C}c + f) + \ldots \nonumber \\
& \qquad \dot{G}_uv_0)) \\
\lambda_{u_1} & \geq 0 \nonumber \\
\lambda_{n_1} & \geq - \hat{M}_n\dot{G}_nv_0 - K_n\phi_{n_0} - (hK_n + B_n)(\dot{\phi}_{n_0} + h (G_n\inv{M}((\tr{G_b}\inv{C}D + I)(\ldots \nonumber \\
& \qquad \tr{G_n}\lambda_{n_1} + \tr{G_r}\lambda_{r_1} + \tr{G_s}\lambda_{s_1} + \tr{G_u}\lambda_{u_1}) + \tr{G_b}\inv{C}c + f) + \ldots \nonumber \\
& \qquad \dot{G}_nv_0)) \\
\lambda_{n_1} & \geq 0 \nonumber \\
\mu^2 \lambda_{n_1}^2 & \geq \lambda_{r_1}^2 + \lambda_{s_1}^2. \nonumber
\end{align}

For convenience, the gradient of the objective function, taken with respect to $\lambda^*_1$ is
\begin{align}
\nabla_{\lambda^*_1} = \begin{bmatrix} \zeta \lambda_{n_1} \\ \lambda_{r_1} \\ \lambda_{s_1} \\ \zeta \lambda_{u_1} \end{bmatrix} + (\begin{bmatrix} 0 \\ I \\ 0 \\ 0 \end{bmatrix} + \tr{E})(\lambda_{r_1} + E\lambda^*_1 + d) + (\begin{bmatrix} 0 \\ 0 \\ I \\ 0 \end{bmatrix} + \tr{F})(\lambda_{s_1} + F\lambda^*_1 + e),
\end{align}
and the Hessian of the objective function is:
\begin{align}
\nabla^2_{\lambda^*_1} = \begin{bmatrix}
\zeta & 0 & 0 & 0 \\
0 & I & 0 & 0 \\
0 & 0 & I & 0 \\
0 & 0 & 0 & \zeta 
\end{bmatrix} +
(\begin{bmatrix} 0 \\ I \\ 0 \\ 0 \end{bmatrix} + \tr{E})(E + \begin{bmatrix} 0 & I & 0 & 0 \end{bmatrix})  +
(\begin{bmatrix} 0 \\ 0 \\ I \\ 0 \end{bmatrix} + \tr{F})(F + \begin{bmatrix} 0 & 0 & I & 0 \end{bmatrix}). \label{eqn:ofn-Hessian}
\end{align}

Also for convenience, I provide the optimization criterion and linear constraints in matrix form.
\begin{tcolorbox}
\begin{align}
\minimize_{\lambda^*_1} & \frac{1}{2}\tr{\lambda^*_1}H\lambda^*_1 + \tr{\lambda^*_1}g \\   
\textrm{subject to } & A\lambda^*_1 \geq q \\
& \lambda_{u_1} \geq 0 \nonumber \\
& \lambda_{n_1} \geq 0 \nonumber \\
& \mu^2 \lambda_{n_1}^2 \geq \lambda_{r_1}^2 + \lambda_{s_1}^2, \nonumber
\end{align}
\end{tcolorbox}
where
\begin{align}
H & \equiv \nabla^2_{\lambda^*_1}, \\
g & \equiv \nabla_{\lambda^*_1}(0) = (\begin{bmatrix} 0 \\ I \\ 0 \\ 0 \end{bmatrix} + \tr{E})d  + (\begin{bmatrix} 0 \\ 0 \\ I \\ 0 \end{bmatrix} + \tr{F})e, \\
A & \equiv \begin{bmatrix} 
\eta G_u\Lambda\tr{G_n} & \eta G_u\Lambda\tr{G_r} & \eta G_u\Lambda\tr{G_s} & \eta G_u\Lambda\tr{G_u} + I \\
I + \chi G_n\Lambda\tr{G_n} & \chi G_n\Lambda\tr{G_r} & \chi G_n\Lambda\tr{G_s} & \chi G_n\Lambda\tr{G_u} \\
\end{bmatrix}, \\  
q & \equiv \begin{bmatrix} -\hat{M}_u\dot{G}_uv_0 - K_u\phi_{u_0} - (hK_u + B_u)(\dot{\phi}_{u_0} + h (G_u\inv{M}(\tr{G_b}\inv{C}c + f) + \dot{G}_uv_0)) \\
    -\hat{M}_n\dot{G}_nv_0 - K_n\phi_{n_0} - (hK_n + B_n)(\dot{\phi}_{n_0} + h (G_n\inv{M}(\tr{G_b}\inv{C}c + f) + \dot{G}_nv_0)) \end{bmatrix}, \\
\eta & \equiv (h^2K_u + hB_u), \\
\chi & \equiv (h^2K_n + hB_n), \\
\Lambda & \equiv \inv{M}(\tr{G_b}\inv{C}D + I),\\
c & \equiv -\hat{M}_b\dot{G}_bv_0 - K_b\phi_{b_0} - (hK_b + B_b)(\dot{\phi}_{b_0} + h (G_b\inv{M}f + \dot{G}_bv_0) \\
D & \equiv -(h^2K_b + hB_b)G_b\inv{M}\\
C & \equiv (I - D\tr{G_b})\\
G^* & \equiv \tr{\begin{bmatrix}\tr{G_n} & \tr{G_r} & \tr{G_s} & \tr{G_u} \end{bmatrix}}\\
\lambda^*_1 & \equiv \tr{\begin{bmatrix} \tr{\lambda_{n_1}} & \tr{\lambda_{r_1}} & \tr{\lambda_{s_1}} & \tr{\lambda_{u_1}} \end{bmatrix}},\\
E & \equiv (h^2K_r + hB_r)G_r\inv{M}(\tr{G_b}\inv{C}D + I)\tr{G^*}, \\
F & \equiv (h^2K_s + hB_s)G_s\inv{M}(\tr{G_b}\inv{C}D + I)\tr{G^*}, \\
d & \equiv \hat{M}_r\dot{G}_rv_0 + K_r\phi_{r_0} + (hK_r + B_r)(\dot{\phi}_{r_0} + h(G_r\inv{M}f + \dot{G}_rv_0)), \\
e & \equiv \hat{M}_s\dot{G}_sv_0 + K_s\phi_{s_0} + (hK_s + B_s)(\dot{\phi}_{s_0} + h(G_s\inv{M}f + \dot{G}_sv_0)).
\end{align}

We can prove \emph{strict convexity} of this QCQP by examining \eqref{eqn:ofn-Hessian}. Both the second and third matrices are positive semi-definite: the second matrix, for example, is a sum of terms $\tr{E}E$ (clearly positive semi-definite, at least), the zero matrix with positive definite submatrix $I$ (clearly positive semi-definite), and a positive semi-definite sub-matrix of $E$. The three terms in \eqref{eqn:ofn-Hessian} thus represent a positive definite, a positive semi-definite, and a positive semi-definite matrix, respectively, which implies that their sum, the Hessian, is a positive definite matrix. Therefore, this QCQP is strictly convex.

Additionally, since $\lambda_{u_1} = \lambda_{n_1} = \infty, \lambda_{r_1} = \lambda_{s_1} = 0$ is an initial feasible point, it is trivial to produce an initial iterate (finding such a point would normally require solving a secondary quadratic program or linear program~\cite{Nocedal:2006}). An approximation to the minimum can be found using an anytime algorithm.

\bibliographystyle{abbrv}
\bibliography{paper}

\end{document}